\documentclass[12pt]{amsart}

\usepackage{amssymb, amscd, txfonts}
\usepackage{graphicx}
\usepackage[all]{xy} 
\usepackage{array, booktabs}
\usepackage{tabularx} 

\numberwithin{equation}{section}

\newtheorem{theorem}{Theorem}[section]
\newtheorem{proposition}[theorem]{Proposition}
\newtheorem{lemma}[theorem]{Lemma}

\theoremstyle{definition}

\newtheorem{example}[theorem]{Example}

\theoremstyle{remark}
\newtheorem{remark}[theorem]{Remark}

\newtheorem{step}{Step}

\renewcommand{\hom}{\operatorname{Hom}}

\newcommand{\Z}{\mathbb{Z}}
\newcommand{\Q}{\mathbb{Q}}
\newcommand{\R}{\mathbb{R}}
\newcommand{\C}{\mathbb{C}}

\newcommand{\proj}{{\mathbb P}}

\newcommand{\D}{\mathcal{D}}
\newcommand{\G}{\Gamma}
\newcommand{\Res}{\textrm{Res}}
\newcommand{\Kp}{\mathsf{K}_{p,X}^{\bullet}}
\newcommand{\Mp}{\mathsf{M}_{p,{\G}}^{\bullet}}
\newcommand{\CH}{\textrm{CH}}
\newcommand{\OL}{{\rm O}^{+}(L)}
\newcommand{\OLd}{{\rm O}^{+}(L')}
\newcommand{\Otilde}{\tilde{{\rm O}}^{+}}
\newcommand{\Ohat}{\hat{{\rm O}}^{+}}
\newcommand{\GL}{\Gamma_{L}}

\begin{document}

\title[]{Borcherds products approximating Gersten complex}
\author[]{Shouhei Ma}
\thanks{Supported by KAKENHI 21H00971} 
\address{Department~of~Mathematics, Science~Institute~of~Tokyo, Tokyo 152-8551, Japan}
\email{ma@math.titech.ac.jp}

\begin{abstract}
For an orthogonal modular variety, 
we construct a complex which is defined in terms of lattices and elliptic modular forms, 
which resembles the Gersten complex in Milnor $K$-theory, 
and which has a morphism to the Gersten complex of the modular variety by the Borcherds lifting. 
This provides a formalism for approaching 
the higher Chow groups of the modular variety 
by special cycles and Borcherds products. 
The construction is an incorporation of 
the theory of Borcherds products and ideas from Milnor $K$-theory. 
\end{abstract}

\maketitle

\section{Introduction}\label{sec: intro}


The study of cohomology of modular varieties 
started with the works of Eichler, Shimura and Matsushima who discovered 
how to connect modular and automorphic forms to cohomology. 
Since then, this subject has grown rich and is now related to various branches of Mathematics. 

On the other hand, the development of algebraic geometry led to the discovery of motivic cohomology 
as a sort of generalized cohomology theory for algebraic varieties. 
Geometrically, motivic cohomology is realized as \textit{higher Chow groups} (\cite{Bl}). 
These groups are in general mysterious, 
but expected to encode deep arithmetic and geometric significance. 

If one believes in the unity of Mathematics, 
it is then natural to expect that higher Chow groups of modular varieties would be rich objects. 
After the influential work of Beilinson \cite{Be} on modular curves, 
the study in this direction has centered around low-dimensional cases, 
such as Hilbert modular surfaces (e.g., \cite{Ra1}, \cite{Ra2}, \cite{Ra3}). 
It seems rather recent that the higher-dimensional case is starting to be investigated. 


The purpose of this paper is to construct a path from modular forms to algebraic $K$-theory 
for orthogonal modular varieties by using the theory of Borcherds products \cite{Bo95}, \cite{Bo98}.  
To set up the notations, let $L_0$ be an even lattice of signature $(2, n)$ with $n>0$ 
and $X={\D}/{\G}$ be a modular variety associated to $L_0$, 
where ${\D}$ is the Hermitian symmetric domain 
and ${\G}$ is a suitable subgroup of the orthogonal group of $L_0$. 
The \textit{Borcherds lift} is a multiplicative lifting $f\mapsto \psi_{L_0}(f)$ 
from certain (weakly holomorphic, vector-valued) modular forms $f$ of one variable 
to meromorphic modular forms $\psi_{L_0}(f)$ on $X$ with some remarkable properties. 
(We use the arithmetic normalization of \cite{HMP}.) 
The same construction also applies to sublattices $L\subset L_0$ of signature $(2, \ast)$ 
and produces Borcherds products on 
the \textit{special cycle} $Z_L\subset X$, i.e., the sub modular variety defined by $L$. 

Among the properties of Borcherds products, 
the most relevant to us is that its divisor is an explicit combination of special cycles of codimension $1$. 
This property naturally leads to the idea to use Borcherds products for constructing higher Chow cycles of small degree. 
Although this idea in its primitive form is more or less classical, 
it seems to have been taken up only sporadically (\cite{Ga}, \cite{Sr2}). 
We wish to develop this idea in a more coherent and comprehensive way.

\subsection*{The dictionary}

Recall that a fundamental role in the study of higher Chow groups 
is played by the \textit{Gersten complex} (in Milnor $K$-theory) \cite{Ka}: 
\begin{equation*}\label{eqn: Gersten intro}
{\Kp} \; \: : \; \; \; 
\cdots \longrightarrow \bigoplus_{Z\in X^{(r)}} K_{p-r}^{M}{\C}(Z) 
\longrightarrow \bigoplus_{Z'\in X^{(r+1)}} K_{p-r-1}^{M}{\C}(Z') \longrightarrow \cdots. 
\end{equation*}
Here $X^{(r)}$ is the set of irreducible subvarieties of $X$ of codimension $r$,  
$K_{q}^{M}F=K_{q}^{M}F\otimes_{{\Z}}{\Q}$ is the rational Milnor $K$-group of a field $F$, 
and the boundary maps 
are defined by using the tame symbols. 
The cohomology of ${\Kp}$ is related to the rational higher Chow groups ${\CH}^{p}(X, m)_{{\Q}}$ of $X$ 
via the coniveau spectral sequence. 
In particular, 
$H^{p-m}{\Kp}\simeq {\CH}^{p}(X, m)_{{\Q}}$ 
when $m\leq 2$. 

Now the recursive character of special cycles and Borcherds products 
reminds us of certain essence of Gersten complex. 
We can now explain our goal more specifically: 
to construct a complex which is defined in terms of lattices and elliptic modular forms, 
and which has a morphism to the Gersten complex by the Borcherds lifting. 
The whole idea can be roughly expressed in the following dictionary: 

\begin{center}
\begin{tabular}{|c|c|c|c|c|} 
\hline  
$K$-theory   & $X^{(r)}$  & ${\C}(Z)^{\times}$ & $K_{q}^{M}{\C}(Z)$   & tame symbol  \\ \hline  
MF & $L_{0}^{(r)}$ & $M(L)$ & $\wedge^{q}M(L)$ & quasi-pullback \\ \hline 
\end{tabular}
\end{center}

\noindent
Here $L_{0}^{(r)}$ is the set of ${\G}$-equivalence classes of 
primitive sublattices $L$ of $L_0$ of corank $r$ 
(which correspond to special cycles $Z_L$ of codimension $r$), 
and $M(L)$ is the space of input modular forms for Borcherds products of weight $0$ for $L$, tensored with ${\Q}$. 
Thus $M(L)$ is a countably infinite-dimensional ${\Q}$-linear space of elliptic modular forms. 

A key point in our construction is to define the residue map 
\begin{equation}\label{eqn: residue intro}
\wedge^{q}M(L) \to \wedge^{q-1}M(L') 
\end{equation}
for $L'\subset L$ by using combination of \textit{quasi-pullback}. 
This is a renormalized restriction operator for Borcherds products (\cite{Bo95}, \cite{BKPSB}); 
on the level of input forms, it is given by a linear map, 
essentially multiplication by the theta series of 
$(L')^{\perp}\cap L$ (\cite{Ma}, \cite{Ze}). 
In this way, 
by translating ideas from Milnor $K$-theory, 
we can form the sequence 
\begin{equation*}\label{eqn: Gersten intro}
{\Mp} \: \; : \; \; \; 
\cdots \longrightarrow \bigoplus_{L\in L_{0}^{(r)}} \wedge^{p-r}M(L)^{G_L} 
\stackrel{\partial}{\longrightarrow} \bigoplus_{L'\in L_{0}^{(r+1)}} \wedge^{p-r-1}M(L')^{G_{L'}} \longrightarrow \cdots 
\end{equation*}
whose terms are defined in terms of lattices and elliptic modular forms, 
and whose boundary maps come from multiplication by theta series. 
(In practice, we need to take the invariant part for a certain finite group $G_L$, but this is a technical matter.) 


Our main result can be summarized as follows. 

\begin{theorem}[Theorem \ref{thm: main}]\label{thm: main intro}
Let $p\leq n$. 
The sequence ${\Mp}$ is a complex, i.e., $\partial \circ \partial =0$. 
The Borcherds lift maps 
\begin{equation}\label{eqn: Borcherds to Milnor}
\wedge^{q}M(L)^{G_L} \to K^{M}_{q}{\C}(Z_L), \quad 
f_1\wedge \cdots \wedge f_q \mapsto 
\{ \psi_L(f_1), \cdots, \psi_L(f_q) \} 
\end{equation}
for sublattices $L$ of $L_0$ define a morphism 
\begin{equation*}
{\Mp} \to {\Kp}
\end{equation*}
of complexes. 
\end{theorem} 

In particular, we obtain natural maps 
\begin{equation}\label{eqn: Borcherds to Chow intro}
H^{p-m}{\Mp} \to {\CH}^{p}(X, m)_{{\Q}} 
\end{equation}
for $0\leq m \leq 2$. 
(The case $m=0$ is just the usual map sending special cycles to the Chow group.) 
When $m=1$, it will be more natural to pass to the indecomposable part of ${\CH}^{p}(X, 1)_{{\Q}}$: 
\begin{equation}\label{eqn: ind part intro}
H^{p-1}{\Mp} \to {\CH}^{p}(X, 1)_{{\rm ind},{\Q}}, 
\end{equation}
as the image of 
$\mathsf{M}_{p,{\G}}^{p-1} \to \mathsf{K}_{p,X}^{p-1}$ 
usually does not contain apparent decomposable cycles. 
%
When $m\geq 3$, $H^{p-m}{\Mp}$ gets related to ${\CH}^{p}(X, m)_{{\Q}}$ more indirectly by 
\begin{equation*}
H^{p-m}{\Mp} \to H^{p-m}{\Kp} \leftarrow {\CH}^{p}(X, m)_{{\Q}}, 
\end{equation*}
where the right map comes from the coniveau spectral sequence. 

In this way, Theorem \ref{thm: main intro} provides a formalism for 
approaching the higher Chow groups of $X$ via lattices and elliptic modular forms. 
In particular, in the case of $(2, 2)$-cycles with $n\geq 2$, 
the construction of $\mathsf{M}_{2,{\G}}^{\bullet}$ provides an answer to a question of Sreekantan (\cite{Sr2} \S 6.0.2).  
Although the complex ${\Mp}$ is still large, 
it would be at least more manageable than the original Gersten complex ${\Kp}$. 
It is desirable to develop techniques of computation with ${\Mp}$.

\subsection*{Related works}

In his pioneering paper (\cite{Ra3} \S 7.14), 
Ramakrishnan proposed a "modular complex": 
a subsequence of the Gersten complex where only modular subvarieties are considered. 
However, his sequence as it stands would not be a complex; 
the usual proof using the Weil reciprocity (\cite{Ka}, \cite{GS}) 
breaks down if one excludes general subvarieties while keeping the $K$-groups the same. 
From this point of view, the complex ${\Mp}$ can be regarded as a ``more modular'' 
modification of Ramakrishnan's sequence where 
the $K$-groups and boundary maps are also replaced by modular stuff.  
Thus the property of ${\Mp}$ being a complex would be indeed nontrivial. 

Our substitution of $\wedge^{q}$ in place of $K_{q}^{M}$ 
can be considered as an approximation of $K$-theory 
which compensates for the lack of summation of Borcherds products. 
The corresponding sequence   
\begin{equation*}
\cdots \longrightarrow \bigoplus_{Z\in X^{(r)}} \wedge^{p-r}{\C}(Z)^{\times}_{\Q}  
\longrightarrow \bigoplus_{Z'\in X^{(r+1)}} \wedge^{p-r-1}{\C}(Z')^{\times}_{\Q} \longrightarrow \cdots  
\end{equation*}
shows up in the polylogarithmic motivic complex of Goncharov \cite{Go} (which is defined when $p\leq 3$). 
Hence it would be also interesting to compare ${\Mp}$ with the Goncharov complex. 

\subsection*{Further questions}

Two fundamental but difficult questions that arise from Theorem \ref{thm: main intro} are the following: 
\begin{enumerate}
\item To examine whether the kernel and cokernel of the map \eqref{eqn: Borcherds to Chow intro} or \eqref{eqn: ind part intro} 
are finite-dimensional, especially in the case $p=2$. 
\item To examine whether the cohomology of ${\Mp}$ is finite-dimensional 
(as was asked by Ramakrishnan \cite{Ra3} for his sequence). 
\end{enumerate}
The question (1) is motivated by 
the theorem of Bruinier \cite{Br} and Bergeron-Li-Millson-Moeglin \cite{BLMM} 
that the map \eqref{eqn: Borcherds to Chow intro} is typically an isomorphism 
in the case of divisors. 
A natural approach toward understanding the image is to compose these maps with the regulator map. 
Garcia \cite{Ga} studied the naive regulators of $(2, 1)$-cycles of this type from the viewpoint of theta lifting. 
As for the question (2), 
a result of Bruinier-Raum-Zhang (\cite{BR}, \cite{Zhang}) says that 
the image of $H^p{\Mp}\to {\rm CH}^{p}(X)_{{\Q}}$ has finite dimension for any $p$. 
Thus, for $m=0$, a version of Ramakrishnan's question indeed has an affirmative answer. 
The case $m>0$ would require a new idea. 

Finally, we should mention two potential directions of application: 
\begin{enumerate}
\item To approach the Beilinson conjecture for orthogonal modular varieties 
by producing enough special higher cycles via modular form computation. 
\item To set up a stage for investigating generating series with coefficients in special higher cycles, 
namely a higher Chow version of Kudla's program.  
\end{enumerate}
It would be interesting if special higher cycles could provide a forum where 
the Beilinson conjecture and the Kudla program interact. 

\subsubsection*{Organization}

This paper is organized as follows. 
\S \ref{sec: Borcherds} is a recollection of Borcherds products. 
\S \ref{sec: Gersten} is a recollection of Milnor $K$-theory. 
In \S \ref{sec: Borcherds to Milnor}, we define the Borcherds lift map \eqref{eqn: Borcherds to Milnor}. 
In \S \ref{sec: residue}, we define the residue map \eqref{eqn: residue intro}. 
In \S \ref{sec: BG cpx}, we define the sequence ${\Mp}$. 
In \S \ref{ssec: complex}, we prove Theorem \ref{thm: main intro}. 
\S \ref{sec: regulator} is a collection of remarks on the regulators of special higher cycles.  

\subsubsection*{Acknowledgement}

It is my pleasure to thank Shuji Saito for instruction in $K$-theory.


\section{Borcherds products}\label{sec: Borcherds}

In this section we recall a basic theory of orthogonal modular varieties and Borcherds products 
following \cite{Bo98} and \cite{Br}.

\subsection{Orthogonal modular varieties}\label{ssec: modular variety}

Let $L$ be an even lattice of signature $(2, n)$ with $n>0$. 
By this we mean a free ${\Z}$-module of rank $n+2$ equipped with 
a nondegenerate symmetric form $( \cdot , \cdot): L\times L \to {\Z}$ of signature $(2, n)$ 
such that $(l, l)\in 2{\Z}$ for every $l\in L$. 
The orthogonal group of $L$ is denoted by ${\rm O}(L)$. 
The space 
\begin{equation*}
\{ \: {\C}\omega \in {\proj}(L\otimes {\C}) \: | \: (\omega, \omega)=0, \; (\omega, \bar{\omega})>0 \: \} 
\end{equation*}
consists of two connected components. 
We choose one of them and denote it by ${\D}={\D}_{L}$. 
This is the Hermitian symmetric domain associated to $L$. 
We denote by ${\OL}$ the subgroup of ${\rm O}(L)$ preserving ${\D}$. 
If ${\G}$ is a subgroup of ${\OL}$ of finite index, 
the quotient 
\begin{equation*}
X_{{\G}} = {\D}/{\G} 
\end{equation*}
has the structure of an irreducible normal quasi-projective variety of dimension $n$. 
This is called an \textit{orthogonal modular variety}. 

Let $\mathcal{O}(-1)=\mathcal{O}_{{\proj}L_{{\C}}}(-1)|_{{\D}}$ 
be the tautological line bundle on ${\D}$. 
Let $\chi$ be a character of ${\G}$ of finite order. 
A ${\G}$-invariant meromorphic section of $\mathcal{O}(-k)\otimes \chi$ over ${\D}$ 
is called a \textit{meromorphic modular form} of weight $k$ and character $\chi$ with respect to ${\G}$. 
When $n\leq 2$ and $L$ has Witt index $n$, the cusp condition is imposed as usual; 
in other cases this is automatically satisfied by the Koecher principle. 
In the case $k=0$ and $\chi=1$, 
modular functions descend to rational functions on the modular variety $X_{{\G}}$. 
We will be mainly concerned with this case in this paper.

\subsection{Special cycles}\label{ssec: special cycle}

For compatibility with \S \ref{sec: BG cpx}, 
we reset our notation in this subsection: 
we take an even lattice $L_0$ of signature $(2, n)$ with $n>0$, 
a finite-index subgroup ${\G}$ of ${\rm O}^{+}(L_{0})$, 
and write $X=X_{{\G}}={\D}_{L_0}/{\G}$. 

We call a sublattice $L$ of $L_{0}$ a \textit{sublattice of corank} $r$ 
if it is primitive and has signature $(2, n-r)$. 
For such $L$, we denote by ${\rm Stab}_{{\G}}(L)$ the stabilizer of $L$ in ${\G}$. 
The natural map ${\rm Stab}_{{\G}}(L)\to {\OL}$ has finite kernel and cokernel. 
We write ${\GL}<{\OL}$ for its image. 

Let $Z_{L}\subset X$ be the image of ${\D}_{L}\subset {\D}_{L_{0}}$ by the projection ${\D}_{L_0}\to X$. 
This is an irreducible subvariety of $X$. 
We call $Z_{L}$ the \textit{special cycle} associated to $L$. 
We have a natural birational morphism 
\begin{equation*}
\pi : X_{{\GL}}\to Z_{L} \subset X 
\end{equation*}
from the modular variety defined by ${\GL}$. 
This gives the normalization of $Z_{L}$. 
The locus where $\pi$ is not isomorphic is caused by those 
$\gamma\in {\G}- {\rm Stab}_{{\G}}(L)$  
such that $\gamma {\D}_{L}\cap {\D}_{L}\ne \emptyset$, 
or equivalently, 
$L':=\gamma L \cap L$ has signature $(2, \ast)$. 
Thus the non-normal locus of $Z_{L}$ is the union of such $Z_{L'}\subset Z_{L}$.

\subsection{Borcherds products}\label{ssec: Borcherds}

Again let $L$ be an even lattice of signature $(2, n)$ with $n>0$. 
We embed $L$ in the dual lattice $L^{\vee}={\hom}_{{\Z}}(L, {\Z})$ in the natural way. 
The quotient $A_{L}=L^{\vee}/L$ is called the \textit{discriminant group} of $L$. 
It inherits a natural quadratic form $A_L\to {\Q}/2{\Z}$. 
The kernel of the reduction map ${\OL}\to {\rm O}(A_{L})$ is denote by ${\Otilde}(L)$. 
We will use the group 
\begin{equation*}
{\Ohat}(L) := \langle {\Otilde}(L), -{\rm id}_{L} \rangle 
\end{equation*}
frequently, and reserve the notation 
\begin{equation}\label{eqn: tildeXL}
\tilde{X}_{L} := {\D}_{L}/ {\Ohat}(L) = {\D}_{L}/{\Otilde}(L) 
\end{equation}
for the modular variety defined by this group. 

Let ${\C}A_L$ be the group ring of $A_{L}$, and 
${\rm Mp}_{2}({\Z})$ be the metaplectic double cover of ${\rm SL}_{2}({\Z})$. 
There is a certain representation $\rho_{L}$ of ${\rm Mp}_{2}({\Z})$ on ${\C}A_{L}$, 
called the \textit{Weil representation}. 
We refer to \cite{Bo98}, \cite{Br} for the precise definition of $\rho_{L}$ 
(which we will not use explicitly in this paper). 
A meromorphic modular form $f(\tau)$ of weight $1-n/2$ for ${\rm Mp}_{2}({\Z})$ with values in $\rho_{L}$ 
is called \textit{weakly holomorphic} if it is holomorphic on the upper half plane 
(but allowed to be meromorphic at the cusp $i\infty$). 
It has a Fourier expansion of the form 
\begin{equation*}
f(\tau) = 
\sum_{\lambda\in {A_L}}\sum_{n\in (\lambda, \lambda)/2+{\Z}} 
c(\lambda, n)q^n \mathbf{e}_{\lambda}, 
\end{equation*}
where $q^n=\exp(2\pi i n\tau)$ for $n\in {\Q}$ and 
$\mathbf{e}_{\lambda}$ is the basis vector of ${\C}A_L$ corresponding to $\lambda\in A_{L}$. 
The invariance under the center of ${\rm Mp}_{2}({\Z})$ implies that 
$f$ actually takes values in the invariant part of ${\C}A_L$ for $-{\rm id}_{L}$. 
The polynomial 
\begin{equation*}
\sum_{\lambda}\sum_{n<0} c(\lambda, n)q^n \mathbf{e}_{\lambda} 
\end{equation*}
is called the \textit{principal part} of $f$. 
We say that $f$ has rational (resp.~ integral) principal part 
if the coefficients $c(\lambda, n)$ of its principal part are contained in ${\Q}$ (resp.~ ${\Z}$). 
When $n\geq 3$, rationality of the principal part implies $c(0, 0)\in {\Q}$ (\cite{Br} Remark 3.23). 
In general, we denote by $M(L)_{\ast}$ the space of weakly holomorphic modular forms 
of weight $1-n/2$ and type $\rho_{L}$ for ${\rm Mp}_{2}({\Z})$ with rational principal part and with $c(0, 0)\in {\Q}$. 

Let $f=\sum_{\lambda,n}c(\lambda, n)q^n \mathbf{e}_{\lambda}$ 
be a modular form in $M(L)_{\ast}$ with integral principal part and with $c(0, 0)\in 2{\Z}$. 
From such $f$, Borcherds \cite{Bo98} constructed a meromorphic modular form $\psi(f)$ on ${\D}_{L}$ 
with some remarkable properties. 
In particular, $\psi(f)$ is modular for ${\Ohat}(L)$, 
of weight $c(0, 0)/2$ with a character of finite order, 
and its divisor is an explicit combination of special divisors. 
More specifically, for a vector $l\in L$ with $(l, l)<0$, 
the order of $\psi(f)$ along ${\D}_{l^{\perp}}\subset {\D}_{L}$ is given by 
\begin{equation}\label{eqn: nu(l)}
\tilde{\nu}_{l}(f) := 
\sum_{\begin{subarray}{c} \alpha>0 \\ \alpha l\in L^{\vee}\end{subarray}} 
c(\alpha l, \alpha^{2}(l, l)/2). 
\end{equation}
The modular form $\psi(f)$ is called the \textit{Borcherds product} associated to $f$. 

We will be mainly concerned with the case $c(0, 0)=0$. 
In this case, after multiplying $f$ by a suitable natural number, 
$\psi(f)$ is modular of weight $0$ with trivial character, 
and hence descends to a rational function on $\tilde{X}_{L}$. 
We denote by 
\begin{equation}\label{eqn: ML0}
M(L) \subset M(L)_{\ast} 
\end{equation}
the linear subspace defined by the condition $c(0, 0)=0$. 

In fact, in the original definition of $\psi(f)$ (\cite{Bo98} \S 13), 
there was an ambiguity of multiplication by a constant of absolute value $1$. 
In order to get rid of this, we use the arithmetic normalization of Howard and Madapusi Pera (\cite{HMP} \S 1.2 and \S 9). 
In the case $c(0, 0)=0$, 
their result says that after multiplying $f$ by a suitable natural number, 
we can choose the constant so that 
$\psi(f)$ is defined over ${\Q}$ with respect to the underlying Shimura variety. 
(When $n\geq 6$ and $L$ contains the integral hyperbolic plane, 
this can be understood as rationality of the Fourier coefficients: see \cite{HMP} \S 9.1.) 
This rationality determines $\psi(f)$ uniquely up to $\pm 1$, 
rational numbers of absolute value $1$. 
In what follows, we shall use the notation $\psi(f)$ for this arithmetically normalized Borcherds products. 
In this notation, $\psi(f)$ is uniquely determined up to $\pm 1$ and satisfies 
\begin{equation}\label{eqn: multiplicative}
\psi(f+g) = \pm \, \psi(f) \cdot \psi(g), 
\end{equation}
\begin{equation}\label{eqn: equivariant}
\psi(\gamma f) = \pm \, \gamma \psi(f), \quad \gamma \in {\rm O}^{+}(L).  
\end{equation}
Here ${\rm O}^{+}(L)$ acts on $f$ by its natural action on $A_{L}$. 

Of course we should not forget that 
the given form $f$ is multiplied by some natural number 
and that the ambiguity of $\pm 1$ remains. 
But these issues can be ignored if we take the tensor product with ${\Q}$. 
We will resume this discussion in \S \ref{sec: Borcherds to Milnor}.

\subsection{Quasi-pullback}\label{ssec: quasi-pullback}

Let $\psi_L(f)$ be a Borcherds product on ${\D}_{L}$. 
Let $L'$ be a sublattice of $L$ of corank $r<n$. 
We write $K=(L')^{\perp}\cap L$. 
Borcherds (\cite{Bo95}, \cite{BKPSB}) found an operation of renormalized restriction 
of $\psi_{L}(f)$ to ${\D}_{L'}\subset {\D}_{L}$, 
called \textit{quasi-pullback}. 
This is defined as 
\begin{equation*}
\psi_{L}(f) ||_{{\D}_{L'}} := \left. \frac{\psi_{L}(f)}{\prod_{\pm l}(\cdot, l)^{\tilde{\nu}_l}} \right|_{{\D}_{L'}} 
\end{equation*}
Here $\pm l$ ranges over all primitive vectors of $K$ up to $\pm 1$, 
$(\cdot, l)$ is the linear form defined by the pairing with $l$, 
and $\tilde{\nu}_{l}=\tilde{\nu}_{l}(f)$ is the order of $\psi_{L}(f)$ along ${\D}_{l^{\perp}}$ given by \eqref{eqn: nu(l)}. 
When ${\D}_{L'}$ is not contained in the divisor of $\psi_{L}(f)$, this is the ordinary restriction. 

It is proved in \cite{Ma}, \cite{Ze} that 
$\psi_{L}(f) ||_{{\D}_{L'}}$ is a Borcherds product on ${\D}_{L'}$ up to constant, 
with the input modular form given by 
\begin{equation}\label{eqn: |_L'}
f|_{L'} := \langle f\! \uparrow_{L}^{L'\oplus K}, \, \Theta_{K(-1)} \rangle. 
\end{equation}
Here 
$\uparrow_{L}^{L'\oplus K}\colon {\C}A_L\to {\C}A_{L'}\otimes {\C}A_{K}$ 
is a pullback map and 
$\langle \cdot , \, \Theta_{K(-1)} \rangle$ is the contraction with the theta series of $K(-1)$. 
In this paper we will not quite need the precise form of $f|_{L'}$, 
for which we refer to \cite{Ma}. 
Instead what matters for us is that 
the quasi-pullback operation is induced by a \textit{linear} map 
\begin{equation*}
|_{L'} : M(L)_{\ast} \to M(L')_{\ast} 
\end{equation*}
between the spaces of input modular forms. 

\begin{lemma}\label{lem: pullback weight 0}
Let $f\in M(L)$. 
If ${\D}_{L'}$ is not contained in ${\rm div}(\psi_{L}(f))$, 
then 
$\psi_{L}(f)|_{{\D}_{L'}} = \pm \psi_{L'}(f|_{L'})$. 
\end{lemma}

\begin{proof}
We already know that 
$\psi_{L}(f)|_{{\D}_{L'}} = C \cdot \psi_{L'}(f|_{L'})$ 
for some constant $C$. 
What has to be checked is $C=\pm 1$. 
Since both $\psi_{L}(f)|_{{\D}_{L'}}$ and $\psi_{L'}(f|_{L'})$ are defined over ${\Q}$, 
we have $C\in {\Q}$. 
It remains to verify $|C|=1$. 

Let $\Phi_{L}(f)$ be the regularized theta lift of $f$ as in \cite{Bo98} \S 13. 
By the assumption $c(0, 0)=0$, the Borcherds product $\psi_{L}(f)$ satisfies 
\begin{equation}\label{eqn: regularized theta lift}
-4\log |\psi_{L}(f)| = \Phi_{L}(f) 
\end{equation}
(see \cite{Bo98} Theorem 13.3). 
Since $\psi_{L}(f)||_{{\D}_{L'}}$ is the ordinary restriction by our assumption ${\D}_{L'}\not\subset {\rm div}(\psi_{L}(f))$, 
$\psi_{L'}(f|_{L'})$ has weight $0$ again. 
Hence we have 
\begin{equation*}
-4\log |\psi_{L'}(f|_{L'})| = \Phi_{L'}(f|_{L'}) 
\end{equation*}
similarly. 
On the other hand, a result of Zemel \cite{Ze} says that 
$\Phi_{L}(f)|_{{\D}_{L'}}= \Phi_{L'}(f|_{L'})$. 
Comparing these equalities,  
we find that 
$|\psi_{L}(f)|_{{\D}_{L'}}|= |\psi_{L'}(f|_{L'})|$. 
Hence $|C|=1$. 
\end{proof}

\section{Milnor $K$-theory}\label{sec: Gersten}

In this section we recall Gersten complexes in Milnor $K$-theory. 
Our main references are \cite{SS} Chapters 7--8 and \cite{GS} Chapters 7--8. 

\subsection{Milnor $K$-groups}\label{ssec: Milnor K}

Let $F$ be a field of characteristic $0$ and let $q>0$ be a positive integer. 
The $q$-th Milnor $K$-group $K_{q}^{M}F$ of $F$ is defined as the quotient group 
\begin{equation*}
K_{q}^{M}F := 
\overbrace{F^{\times}\otimes_{\Z} \cdots \otimes_{\Z} F^{\times}}^{q}  / \langle \textrm{Steinberg relations} \rangle, 
\end{equation*}
where a \textit{Steinberg relation} means an element of the form 
$x_1\otimes \cdots \otimes x_q$ such that $x_i+x_j=1$ for some $i\ne j$. 
The image of an element $x_1\otimes \cdots \otimes x_q$ in $K_{q}^{M}F$ 
is denoted by $\{ x_1, \cdots, x_q\}$. 
When $q=0$, we put $K_{0}^{M}F={\Z}$. 

In this paper we will always work rationally, i.e., 
work with the tensor product $K^{M}_{q}F\otimes_{\Z}\Q$. 
Therefore, for the sake of simplicity, we shall abuse notation to write 
\begin{equation*}
K^{M}_{q}F = K^{M}_{q}F\otimes_{\Z}\Q 
\end{equation*}
throughout this paper. 
For example, $K^{M}_{0}F={\Q}$ and $K_{1}^{M}F=F^{\times}\otimes_{\Z}\Q$ 
in our shortened notation. 
Since the Steinberg symbols $\{ x_1, \cdots, x_q \}$ are alternative, 
the natural ${\Q}$-linear map 
\begin{equation}\label{eqn: wedge to K}
\wedge^{q}(F^{\times}\otimes_{\Z}\Q ) \to K^{M}_{q}F, \qquad 
x_1\wedge \cdots \wedge x_{q} \mapsto  \{ x_1, \cdots, x_q\}, 
\end{equation}
is well-defined. 

Let $F/E$ be a finite extension of fields. 
The embedding $E\hookrightarrow F$ induces a natural ${\Q}$-linear map 
$K^{M}_{q}E\hookrightarrow K^{M}_{q}F$. 
There is also the transfer map 
\begin{equation*}
N_{F/E} : K^{M}_{q}F \twoheadrightarrow K^{M}_{q}E 
\end{equation*}
satisfying some standard properties (see \cite{SS} \S 7.3 and \cite{GS} \S 7.3). 
%
When $F/E$ is a Galois extension with Galois group $G$, 
the embedding $K^{M}_{q}E\hookrightarrow K^{M}_{q}F$ gives an isomorphism 
\begin{equation}\label{eqn: G-inv}
K^{M}_{q}E \simeq (K^{M}_{q}F)^{G}. 
\end{equation}
Under this identification, the transfer map is given by 
\begin{equation}\label{eqn: norm Galois}
N_{F/E}(\alpha) = \sum_{g\in G} g\alpha  \: \in (K^{M}_{q}F)^{G}, \quad \alpha\in K^{M}_{q}F   
\end{equation}
(see \cite{Su} Corollary 1.10).   

\subsection{Tame symbol}\label{ssec: tame}

Suppose that the field $F$ is endowed with a (normalized) discrete valuation $v$. 
We denote by $A$ the valuation ring and $k$ the residue field. 
Then there is a unique ${\Q}$-linear map 
\begin{equation*}
\partial_{v} : K^{M}_{q}F \to K^{M}_{q-1}k 
\end{equation*}
satisfying 
\begin{equation}\label{eqn: tame symbol}
\partial_{v} \{ a, u_2, \cdots, u_q\}  = \nu(a) \cdot \{ \bar{u}_2, \cdots, \bar{u}_q \} 
\end{equation}
for every $a\in F^{\times}$ and units $u_2, \cdots, u_q$ of $A$, 
where $\bar{u}$ stands for the image of $u\in A$ in $k$. 
The map $\partial_{v}$ is called the (higher) \textit{tame symbol}. 

An explicit form of $\partial_{v}$ is given in \cite{SS} Equation (7.2.4) as follows. 
We fix an element $\pi$ of $A$ with $\nu (\pi)=1$. 
For $a_1, \cdots, a_q\in F^{\times}$ we write 
$a_i=u_i\pi^{v(a_i)}$ with $u_i$ a unit in $A$. 
Then 
\begin{equation*}
\partial_{v} \{ a_1, a_2, \cdots, a_q\}  = 
\sum_{i=1}^{q}(-1)^{i-1} \nu(a_i) \cdot \{ \bar{u}_1, \cdots, \bar{u}_{i-1}, \bar{u}_{i+1}, \cdots, \bar{u}_q \}. 
\end{equation*}
This is independent of the choice of $\pi$. 
For example, when $q=1$, $\partial_{v}$ is just the valuation map $\nu$.  
When $q=2$, we have 
\begin{equation}\label{eqn: tame q=2}
\partial_{v} \{ a, b\} = \overline{b^{v(a)}/a^{v(b)}}, \qquad a, b\in F^{\times}. 
\end{equation}

\begin{remark}
When $q\geq 2$, Equation (7.2.4) of \cite{SS} 
contains more terms, but they are $2$-torsion in the integral $K$-group 
and hence vanishes in our rational $K$-group. 
This explains the difference of \eqref{eqn: tame q=2} and 
the more familiar formula, namely the lack of $(-1)^{v(a)v(b)}$. 
\end{remark}


\subsection{Gersten complex}\label{ssec: Gersten}

Let $X$ be an irreducible variety over ${\C}$. 
Let $Z'\subset Z$ be irreducible subvarieties of $X$ with $\dim Z' = \dim Z -1$. 
We write ${\C}(Z)$ for the function field of $Z$ (and similarly for $Z'$). 
Then the boundary map 
\begin{equation*}
\partial^{Z}_{Z'} : K^{M}_{q}{\C}(Z) \to K^{M}_{q-1}{\C}(Z')
\end{equation*}
is defined as follows. 
Let $\pi \colon \hat{Z}\to Z$ be the normalization map and 
$\pi^{-1}(Z')=\sum_{\alpha} Z_{\alpha}'$ be the irreducible decomposition of $\pi^{-1}(Z')$. 
Each divisor $Z_{\alpha}'$ of $\hat{Z}$ defines a discrete valuation on ${\C}(Z) = {\C}(\hat{Z})$ 
with residue field ${\C}(Z_{\alpha}')$. 
We denote by 
\begin{equation*}
\partial_{\alpha} : K^{M}_{q}{\C}(Z) \to K^{M}_{q-1}{\C}(Z_{\alpha}') 
\end{equation*}
the associated tame symbol. 
Next, by the projection $Z_{\alpha}'\to Z'$, we can take the transfer map 
\begin{equation*}
N_{\alpha} : K^{M}_{q-1}{\C}(Z_{\alpha}') \to K^{M}_{q-1}{\C}(Z'). 
\end{equation*}
Then we put 
\begin{equation*}
\partial^{Z}_{Z'} := \sum_{\alpha} N_{\alpha} \circ \partial_{\alpha} \; : \; 
K^{M}_{q}{\C}(Z) \to K^{M}_{q-1}{\C}(Z'). 
\end{equation*}

Now let $p>0$. 
The \textit{Gersten complex} of $X$ is defined as (\cite{Ka})
\begin{equation*}
{\Kp} \: : \; \; 
\cdots \longrightarrow \bigoplus_{Z\in X^{(r)}} K_{p-r}^{M}{\C}(Z) 
\stackrel{\partial}{\longrightarrow} \bigoplus_{Z'\in X^{(r+1)}} K_{p-r-1}^{M}{\C}(Z') \longrightarrow \cdots, 
\end{equation*}
where $X^{(r)}$ is the set of irreducible subvarieties of $X$ of codimension $r$ 
and $\partial=(\partial^{Z}_{Z'})$. 
(When $Z'\not\subset Z$, we set $\partial^{Z}_{Z'}=0$.) 
The term indexed by $X^{(r)}$ is placed in degree $r$. 
Thus ${\Kp}$ has length $p+1$, 
starting from $K^{M}_{p}{\C}(X)$ in degree $0$: 
\begin{equation*}
K^{M}_{p}{\C}(X) \longrightarrow \cdots \longrightarrow 
\bigoplus_{Z\in X^{(p-1)}} {\C}(Z)^{\times}_{{\Q}} \longrightarrow \bigoplus_{Z'\in X^{(p)}} {\Q}[Z']. 
\end{equation*} 
It is proved by Kato \cite{Ka} that 
${\Kp}$ is indeed a complex, i.e., $\partial \circ \partial =0$ 
(see also \cite{GS} \S 8.1).

\subsection{Higher Chow groups}\label{ssec: higher Chow}

Let ${\CH}^p(X, m)$ be the higher Chow group of $X$ of codimension $p$ and ($K$-theoretic) degree $m$ 
as introduced by Bloch \cite{Bl}. 
An element of ${\CH}^p(X, m)$ is called a higher Chow cycle of type $(p, m)$, or simply a $(p, m)$-cycle. 
The Gersten complexes are related to the higher Chow groups as follows (\cite{SS} \S 8.5.d). 

For each $p$, we have the coniveau spectral sequence (\cite{Bl} \S 10) 
\begin{equation*}
E_{1}^{r,s} = \bigoplus_{Z\in X^{(r)}} {\CH}^{p-r}({\C}(Z), -r-s) \; \; \Rightarrow \; \; 
E_{\infty}^{r+s} = {\CH}^{p}(X, -r-s). 
\end{equation*}
The $E_1$-page is supported on the range $s\leq -p$. 
Them the Gersten complex ${\Kp}$ is identified with the edge line $E_1^{\bullet, -p}$, 
tensored with ${\Q}$ and up to sign. 
From a partial degeneration of the spectral sequence, 
we obtain natural isomorphisms (\cite{SS} Corollary 8.5.21) 
\begin{equation}\label{eqn: Gersten higher Chow}
H^{p-m}{\Kp} \simeq {\CH}^p(X, m)_{\Q}, \qquad 0\leq m \leq 2. 
\end{equation}
Via this, a $(p, 1)$-cycle on $X$ can be expressed as a formal sum 
\begin{equation*}
\sum_{i}(Z_i, \psi_i), \qquad Z_i\in X^{(p-1)}, \; \; \psi_i\in {\C}(Z_i)^{\times}, 
\end{equation*}
satisfying the cocycle condition $\sum_i {\rm div}(\psi_i)=0$. 
Similarly, a $(p, 2)$-cycle on $X$ can be expressed as a formal sum 
\begin{equation*}
\sum_{i}(Z_i,  \{ \psi_i, \varphi_i \} ), \qquad Z_i\in X^{(p-2)}, \; \; \psi_i, \varphi_i \in {\C}(Z_i)^{\times}, 
\end{equation*}
with the cocycle condition 
$\sum_i \partial^{Z_i}_{Z'} \{ \psi_i, \varphi_i \} = 0$ 
for any $Z'\in X^{(p-1)}$. 

When $m\geq 3$, the Gersten complexes and the higher Chow groups are related by 
\begin{equation}\label{eqn: CH coniveau}
{\CH}^p(X, m)_{\Q} = E_{\infty}^{-m} \twoheadrightarrow 
E_{\infty}^{p-m,-p} \hookrightarrow E_{2}^{p-m,-p} \simeq H^{p-m}{\Kp}. 
\end{equation}
When $m=3$, 
we have $E_{\infty}^{p-m,-p}=E_{2}^{p-m,-p}$, 
so this reduces to 
\begin{equation}\label{eqn: Chow Gersten p=3}
{\CH}^{p}(X, 3)_{\Q} \twoheadrightarrow H^{p-3}{\Kp}. 
\end{equation}

\section{From Borcherds products to Milnor $K$-group}\label{sec: Borcherds to Milnor}

In this section we define the map from 
$\wedge^{q}M(L)$ to the rational Milnor $K$-group of the modular variety, 
and prove some functorial properties. 

\subsection{Borcherds lift map}

Let $L$ be an even lattice of signature $(2, n)$ with $n>0$. 
Recall that the space $M(L)$ of input modular forms for Borcherds products of weight $0$ for $L$ 
(with rational coefficients) is defined in \eqref{eqn: ML0}. 
We write $M(L)_{{\Z}}\subset M(L)$ for the sub $\Z$-module consisting of those $f$ 
with integral principal part and $c(0, 0)\in 2{\Z}$, and such that the arithmetic normalization $\psi(f)$  makes sense. 
We have 
\begin{equation}\label{eqn: ML0Z}
M(L)_{{\Z}}\otimes_{\Z} \Q =M(L). 
\end{equation}
As in \eqref{eqn: tildeXL}, we write 
$\tilde{X}_{L}={\D}_{L}/{\Ohat}(L)$. 

\begin{lemma}\label{lem: Borcherds lift reformulate}
The Borcherds lifting $f\mapsto \psi(f)$ defines a ${\Q}$-linear map 
\begin{equation}\label{eqn: Borcherds lift q=1}
\psi = \psi_{L} : M(L) \to {\C}(\tilde{X}_{L})^{\times} \otimes_{{\Z}}{{\Q}}. 
\end{equation}
This is equivariant with respect to ${\OL}$. 
\end{lemma}

\begin{proof}
Recall that the normalized Borcherds product $\psi(f)$ 
defines a rational function on $\tilde{X}_{L}$ up to the ambiguity of $\pm 1$. 
This ambiguity is a $2$-torsion in ${\C}(\tilde{X}_L)^{\times}$ and 
hence vanishes in ${\C}(\tilde{X}_{L})^{\times}_{\Q}$.
Thus we have the well-defined map 
$M(L)_{{\Z}} \to {\C}(\tilde{X}_{L})^{\times}_{{\Q}}$. 
By \eqref{eqn: multiplicative} and \eqref{eqn: equivariant}, 
this map is ${\Z}$-linear and ${\OL}$-equivariant. 
By \eqref{eqn: ML0Z}, this extends to a ${\Q}$-linear map from $M(L)$. 
\end{proof}

Let $q\geq 0$. 
Composing the exterior product of \eqref{eqn: Borcherds lift q=1} 
with the natural map 
\begin{equation*}
\wedge^{q} {\C}(\tilde{X}_{L})^{\times}_{{\Q}} \to K^{M}_{q} {\C}(\tilde{X}_{L}) 
\end{equation*}
explained in \eqref{eqn: wedge to K}, 
we obtain the ${\rm O}^{+}(L)$-equivariant ${\Q}$-linear map 
\begin{equation}\label{eqn: Borcherds lift general q} 
\psi_{L}  :  \wedge^{q}M(L) \to K^{M}_{q}{\C}(\tilde{X}_L), \quad 
f_1\wedge \cdots \wedge f_q \mapsto \{ \psi(f_1), \cdots, \psi(f_q) \}, 
\end{equation}
where $\psi(f_i)$ is viewed as a rational function on $\tilde{X}_{L}$. 
When $q=0$, this is just the identity of 
${\Q} = \wedge^{0}M(L) = K_0^{M}{\C}(\tilde{X}_{L})$. 

More generally, let ${\G}$ be a subgroup of ${\rm O}^{+}(L)$ containing ${\Ohat}(L)$. 
We set 
\begin{equation*}
G = G_{{\G}} := {\G}/{\Ohat}(L) = 
{\rm Im}({\G}\to {\rm O}(A_L))/\pm 1 
\end{equation*}
and take the $G$-invariant part of \eqref{eqn: Borcherds lift general q}. 
Since 
\begin{equation*}
K^{M}_{q}{\C}(\tilde{X}_L)^{G} = K^{M}_{q}{\C}(\tilde{X}_L/G) = K^{M}_{q}{\C}(X_{{\G}}) 
\end{equation*}
by \eqref{eqn: G-inv} 
where $X_{{\G}}={\D}_{L}/{\G}$, 
this gives a ${\Q}$-linear map 
\begin{equation}\label{eqn: Borcherds lift general q G}
\psi_{{\G}}  :  \wedge^{q}M(L)^{G} \to K^{M}_{q}{\C}(X_{{\G}}). 
\end{equation}
Here we write   
$\wedge^{q}M(L)^{G} = (\wedge^{q}M(L))^{G}$
for simplicity. 
(This should not be confused with $\wedge^{q}(M(L)^{G})$.) 
We call \eqref{eqn: Borcherds lift general q G} the \textit{Borcherds lift map}.

\begin{remark}
Studying the kernel of $\wedge^{q}M(L) \to K^{M}_{q}{\C}(\tilde{X}_L)$ 
will lead us to the quest of Steinberg relation between Borcherds products: 
which two Borcherds products $\psi(f), \psi(g)$ of weight $0$ satisfy 
\begin{equation*}
\psi(f)+\psi(g) \equiv 1 
\end{equation*}  
(or more generally a root of unity). 
Since \textit{sum} of Borcherds products has been rarely considered, this would be a new type of problem. 
\end{remark}

\subsection{Functoriality}

We prove some functorial properties of the Borcherds lift maps. 
We fix the lattice $L$ for a while and let 
$\tilde{{\G}}<{\G}$ be two subgroups of ${\rm O}^{+}(L)$ containing ${\Ohat}(L)$. 
We write $G=G_{{\G}}$, $\tilde{G}=G_{\tilde{{\G}}}$, 
$X=X_{{\G}}$ and $\tilde{X}=X_{\tilde{{\G}}}$. 
Then $\tilde{G}<G$ and we have a natural projection $\tilde{X} \to X$. 
Let $q\geq 0$. 

\begin{lemma}\label{lem: pullback}
The following diagram commutes: 
\begin{equation*}
\xymatrix@C+1pc{
\wedge^{q}M(L)^{G} \ar[r]^{\psi_{{\G}}} \ar@{^{(}->}[d]  & K^{M}_{q}{\C}(X) \ar@{^{(}->}[d]  \\ 
\wedge^{q}M(L)^{\tilde{G}} \ar[r]^{\psi_{\tilde{{\G}}}} & K^{M}_{q}{\C}(\tilde{X}) 
}
\end{equation*}
where $\wedge^{q}M(L)^{G}\hookrightarrow \wedge^{q}M(L)^{\tilde{G}}$ is the natural inclusion 
and $K^{M}_{q}{\C}(X) \hookrightarrow K^{M}_{q}{\C}(\tilde{X})$ 
is the map induced by $\tilde{X} \to X$. 
\end{lemma}

\begin{proof}
This holds because both $\psi_{{\G}}$ and $\psi_{\tilde{{\G}}}$ are defined as 
restriction of 
$\psi_{L} \colon \! \wedge^{q}\! M(L) \to K^{M}_{q}{\C}(\tilde{X}_L)$. 
\end{proof} 

Next we define the transfer map in the modular side by 
\begin{equation*}
N_{\tilde{{\G}}/{\G}} \: : \: \wedge^{q}M(L)^{\tilde{G}} \to \wedge^{q}M(L)^{G}, \quad 
\omega \mapsto \sum_{\gamma\in G/\tilde{G}} \gamma \omega. 
\end{equation*}
Clearly this is surjective.  

\begin{lemma}\label{lem: push}
The following diagram commutes: 
\begin{equation*}
\xymatrix{
\wedge^{q}M(L)^{\tilde{G}} \ar[r]^{\psi_{\tilde{{\G}}}}  \ar@{->>}[d]_{N_{{\tilde{{\G}}/{\G}}}} 
& K^{M}_{q}{\C}(\tilde{X}) \ar@{->>}[d]^{N_{\tilde{X}/X}} \\ 
\wedge^{q}M(L)^{G} \ar[r]^{\psi_{{\G}}}  & K^{M}_{q}{\C}(X)  
}
\end{equation*}
where $N_{\tilde{X}/X}$ is the $K$-theoretic transfer map for $\tilde{X} \to X$. 
\end{lemma}

\begin{proof}
Since both transfer maps satisfy the transitivity and are surjective, 
it suffices to prove the assertion for $\tilde{{\G}}={\Ohat}(L)$. 
Recall that $G$ is the Galois group of $\tilde{X}_{L}\to X$. 
In the case $q=0$, both transfer maps are multiplication by $|G|$ on ${\Q}$. 
Next let $q>0$ and $f_1, \cdots, f_q\in M(L)$. 
By \eqref{eqn: norm Galois}, 
we have the following equalities in $K^{M}_{q}{\C}(\tilde{X}_{L})$: 
\begin{eqnarray*}
 N_{\tilde{X}_{L}/X}(\psi_{L}(f_1\wedge \cdots \wedge f_q))  & = & 
\sum_{\gamma\in G} \gamma \cdot \{ \psi_{L}(f_1), \cdots, \psi_{L}(f_q) \} \\ 
& = & \sum_{\gamma\in G}   \{ \psi_{L}(\gamma f_1), \cdots, \psi_{L}(\gamma f_q) \} \\ 
& = & \sum_{\gamma\in G}  \psi_{L}( \gamma f_1 \wedge \cdots \wedge \gamma f_q ) \\ 
& = & \psi_{L} \left( \sum_{\gamma\in G}  \gamma (f_1 \wedge \cdots \wedge f_q) \right).   
\end{eqnarray*}
By the definition of $\psi_{{\G}}$, this is equal to 
$\psi_{{\G}} ( N_{\tilde{{\G}}/{\G}}(f_1 \wedge \cdots \wedge f_q))$. 
\end{proof}

\begin{remark}
We also have a functoriality with respect to lattices in the following form. 
Let $L_1$ be a finite-index sublattice of $L$. 
Then ${\Ohat}(L_1)\subset {\Ohat}(L)$ inside 
${\rm O}(L_1\otimes {\Q})={\rm O}(L \otimes {\Q})$. 
We have a natural linear map 
$\uparrow^{L_1}_{L}\colon {\C}A_L \hookrightarrow {\C}A_{L_1}$ 
which satisfies 
$\psi_{L_1}(f \! \! \uparrow_{L}^{L_1}) = \psi_{L}(f)$ 
as modular forms on ${\D}_{L_1}={\D}_{L}$ (see \cite{Ma}). 
This property implies that the diagram 
\begin{equation*}
\xymatrix{
\wedge^{q}M(L) \ar[r]^{\psi_{L}} \ar@{^{(}->}[d]_{\uparrow^{L_1}_{L}} & K^{M}_{q}{\C}(\tilde{X}_{L}) \ar@{^{(}->}[d] \\ 
\wedge^{q}M(L_1) \ar[r]^{\psi_{L_1}} & K^{M}_{q}{\C}(\tilde{X}_{L_1}) 
}
\end{equation*}
commutes, where the right vertical map is induced by the projection $\tilde{X}_{L_1}\to \tilde{X}_{L}$.  
\end{remark}


\section{Residue maps}\label{sec: residue}

In this section we define the residue maps in the modular side 
and prove that they correspond to the tame symbols in the $K$-theory side. 

\subsection{Construction}\label{ssec: residue def}

Let $L$ be an even lattice of signature $(2, n)$ with $n>0$ and $L'$ be a sublattice of $L$ of corank $1$. 
Let ${\G}$ be a subgroup of ${\rm O}^{+}(L)$ containing ${\Ohat}(L)$. 
We denote by ${\G}'={\G}_{L'}<{\OLd}$ the image of ${\rm Stab}_{{\G}}(L') \to {\OLd}$. 
By the theory of Nikulin \cite{Ni}, 
we have ${\G}'\supset {\Ohat}(L')$. 
We write $X=X_{{\G}}$ and $X'=X_{{\G}'}$. 
The image of the natural morphism $X'\to X$, i.e., 
the special divisor associated to $L'$, is denoted by $Z$. 
We also write $G=G_{{\G}}$ and $G'=G_{{\G}'}$. 
This setting and notation are fixed throughout \S \ref{sec: residue}. 

For $f\in M(L)_{\ast}$ with integral principal part and $c(0, 0)\in 2{\Z}$, 
we denote by $\tilde{\nu}(f)$ 
the order of the Borcherds product $\psi(f)$ along the divisor ${\D}_{L'}\subset {\D}_{L}$. 
By the formula \eqref{eqn: nu(l)}, this extends to a ${\Q}$-linear map 
$\tilde{\nu} \colon M(L)_{\ast} \to {\Q}$ 
which is expressed as a sum of some coefficients of the principal part. 
Let $r$ be the ramification index of the projection ${\D}_{L}\to X$ along ${\D}_{L'}\to Z$. 
We put 
\begin{equation}\label{eqn: nu}
\nu := r^{-1}\tilde{\nu} \: : \: M(L)_{\ast} \to {\Q}. 
\end{equation}
If $f\in M(L)_{{\Z}}$, then $\nu(f)$ is the order of the Borcherds product $\psi(f)$ 
as a rational function on $X$ along the divisor $Z$. 
Note that $\tilde{\nu}$ does not depend on ${\G}$ but $\nu$ does. 

Let $|_{L'} : M(L)_{\ast} \to M(L')_{\ast}$ be the quasi-pullback map defined in \eqref{eqn: |_L'}. 
For $q>0$ we define the residue map 
\begin{equation*}
{\Res} \: : \; 
\wedge^{q}M(L)_{\ast} \to \wedge^{q-1}M(L')_{\ast} 
\end{equation*}
by 
\begin{eqnarray*}
 & & {\Res}(f_1 \wedge \cdots \wedge f_q)  \\ 
 & := &  
 \sum_{i=1}^{q} (-1)^{i-1}\nu(f_i) \cdot  
 (f_{1}|_{L'}) \wedge \cdots \wedge (f_{i-1}|_{L'}) \wedge (f_{i+1}|_{L'}) \wedge \cdots \wedge (f_{q}|_{L'}). 
\end{eqnarray*}
For example, when $q=1$, this is just the map $\nu\colon M(L)_{\ast}\to {\Q}$. 
When $q=2$, we have 
\begin{equation}\label{eqn: residue q=2}
{\rm Res}(f_1\wedge f_2)  =  
\nu(f_1)f_2|_{L'} - \nu(f_2) f_1|_{L'}  
 =  (\nu(f_1)f_2 - \nu(f_2) f_1)|_{L'} 
\end{equation}
by the linearity of the operation $|_{L'}$. 
Note that the map ${\Res}$ depends on ${\G}$ (up to constant) because $\nu$ does so. 

\begin{proposition}\label{prop: M(L)0 M(L')0}
We have 
${\rm Res}(\wedge^{q}M(L))\subset \wedge^{q-1}M(L')$. 
\end{proposition}

The proof of this proposition is based on the following observation. 

\begin{lemma}\label{lem: M(L)0 M(L')0}
If $f\in M(L)$ and $\nu(f)=0$, then $f|_{L'}\in M(L')$. 
\end{lemma}

\begin{proof}
In \cite{Ma} \S 3, it is proved that the coefficient $c(0, 0)$ of $f|_{L'}$ is equal to 
the sum of that of $f$ and $\tilde{\nu}(f)$ ($=0$ in the present case). 
Hence $f\in M(L)$ implies $f|_{L'}\in M(L')$. 
\end{proof}

\begin{proof}[(Proof of Proposition \ref{prop: M(L)0 M(L')0})]
We prove by induction on $q$. 
We start with the case $q=2$. 
Let $f_1, f_2\in M(L)$. 
We use the expression \eqref{eqn: residue q=2} and write 
$f=\nu(f_1)f_2 - \nu(f_2) f_1$. 
Then we have 
\begin{equation*}
\nu(f) = \nu(f_1)\nu(f_2) - \nu(f_2) \nu(f_1) = 0. 
\end{equation*}
Since $f\in M(L)$, we have 
${\Res}(f_1\wedge f_2) = f|_{L'}\in M(L')$ 
by Lemma \ref{lem: M(L)0 M(L')0}. 
This proves the assertion in the case $q=2$. 

Suppose that our assertion was proved in degree $q-1$, 
and consider the case of degree $q$. 
Let $f_1, \cdots, f_q\in M(L)$. 
In the case $\nu(f_q)=0$, we have 
\begin{equation*}
{\Res}(f_1 \wedge \cdots \wedge f_q) = 
{\Res}(f_1 \wedge \cdots \wedge f_{q-1})\wedge (f_q|_{L'}) 
\end{equation*}
by the definition of ${\Res}$. 
Then 
${\Res}(f_1 \wedge \cdots \wedge f_{q-1})\in \wedge^{q-2}M(L')$ 
by the assumption of induction, 
while $f_q|_{L'}\in M(L')$ by Lemma \ref{lem: M(L)0 M(L')0}. 

It remains to consider the case where $\nu(f_i)\ne 0$ for any $i$. 
We put $\alpha = \nu(f_q)/\nu(f_1)$. 
Then 
\begin{equation*}
f_1 \wedge f_2 \wedge \cdots \wedge f_q = 
f_1 \wedge f_2 \wedge \cdots \wedge (f_q- \alpha f_1) 
\end{equation*}
and $\nu(f_q- \alpha f_1)=0$, 
so we are reduced to the first case. 
\end{proof}

By Proposition \ref{prop: M(L)0 M(L')0}, we obtain the ${\Q}$-linear map 
\begin{equation*}
{\Res} = {\Res}^{L}_{L'} : \wedge^{q}M(L) \to \wedge^{q-1}M(L'). 
\end{equation*}
Since the map 
$|_{L'}\colon M(L)_{\ast} \to M(L')_{\ast}$ is equivariant for ${\rm Stab}_{{\G}}(L')$ 
and the map $\nu \colon M(L)_{\ast} \to {\Q}$ is invariant for ${\rm Stab}_{{\G}}(L')$, 
we see that the map ${\Res}$ is equivariant for ${\rm Stab}_{{\G}}(L')$. 
Hence it restricts to 
\begin{equation}\label{eqn: residue inv}
{\Res} : \wedge^{q}M(L)^{G} \to \wedge^{q-1}M(L')^{G'} 
\end{equation}
between the relevant invariant parts. 

\subsection{Residue maps and tame symbols}\label{ssec: residue tame}

We prove that 
residue maps correspond to tame symbols by the Borcherds lift map. 

\begin{proposition}\label{prop: residue and tame}
The diagram 
\begin{equation}\label{diagram; residue}
\xymatrix{
\wedge^{q}M(L)^{G} \ar[r]^{\psi_{{\G}}} \ar[d]_{{\rm Res}} & K^{M}_{q}{\C}(X) \ar[d]^{\partial} \\ 
\wedge^{q-1}M(L')^{G'} \ar[r]^-{\psi_{{\G}'}} & K^{M}_{q-1}{\C}(X') 
}
\end{equation}
commutes, 
where $\partial$ is the tame symbol for $Z\subset X$. 
\end{proposition}

\begin{proof}
We take an element 
$\omega=\sum_{\alpha}f_{1}^{(\alpha)}\wedge \cdots \wedge f_{q}^{(\alpha)}$ 
of $\wedge^{q}M(L)^{G}$ 
where $f_{i}^{(\alpha)}\in M(L)$. 
As in the proof of Proposition \ref{prop: M(L)0 M(L')0}, 
we may assume that 
$\nu(f_{i}^{(\alpha)})=0$ for any $i>1$ and $\alpha$. 
Recall that $X=\tilde{X}_{L}/G$ and $X'=\tilde{X}_{L'}/G'$, 
where $\tilde{X}_{L}$ and $\tilde{X}_{L'}$ are as in \eqref{eqn: tildeXL}. 
Let $\tilde{Z}\subset \tilde{X}_{L}$ be the image of 
$\tilde{X}_{L'}\to \tilde{X}_{L}$. 
We have the finite maps $\tilde{X}_{L'}\to \tilde{Z}\to Z$, 
and correspondingly, 
\begin{equation*}
K^{M}_{q-1}{\C}(X') \; \subset \; K^{M}_{q-1}{\C}(\tilde{Z}) \; \subset \; K^{M}_{q-1}{\C}(\tilde{X}_{L'}). 
\end{equation*}
Then, in this $K$-group, we have 
\begin{eqnarray*}
\psi_{{\G}'}({\Res}(\omega)) 
& = & 
\psi_{L'} \left( \sum_{\alpha}  \nu(f_1^{(\alpha)}) \cdot (f_2^{(\alpha)}|_{L'}) \wedge \cdots \wedge (f_q^{(\alpha)}|_{L'}) \right) \\ 
& = & 
\sum_{\alpha}  \nu(f_1^{(\alpha)}) \cdot \{ \psi_{L'}(f_2^{(\alpha)}|_{L'}), \cdots,  \psi_{L'}(f_q^{(\alpha)}|_{L'}) \} \\ 
& = & 
\sum_{\alpha}  \nu(f_1^{(\alpha)}) \cdot \{ \psi_{L}(f_2^{(\alpha)})|_{\tilde{Z}}, \cdots,  \psi_{L}(f_q^{(\alpha)})|_{\tilde{Z}} \},  
\end{eqnarray*}
where the last equality follows from Lemma \ref{lem: pullback weight 0}. 
Let 
$\tilde{\partial} \colon K_{q}^{M}{\C}(\tilde{X}_{L}) \to K_{q-1}^{M}{\C}(\tilde{Z})$ 
be the tame symbol for $\tilde{Z}\subset \tilde{X}_{L}$ 
and $\tilde{r}$ be the ramification index of ${\D}_{L}\to \tilde{X}_{L}$ along ${\D}_{L'}\to \tilde{Z}$. 
Then $r/\tilde{r}$ is the ramification index of $\tilde{X}_{L}\to X$ along $\tilde{Z}\to Z$. 
Since $(r/\tilde{r})\cdot \nu(f_{1}^{(\alpha)})$ is equal to the order of 
$\psi_{L}(f_{1}^{(\alpha)})$ along $\tilde{Z}\subset \tilde{X}_{L}$, 
we have 
\begin{equation*}
\nu(f_1^{(\alpha)}) \cdot \{ \psi_{L}(f_2^{(\alpha)})|_{\tilde{Z}}, \cdots,  \psi_{L}(f_q^{(\alpha)})|_{\tilde{Z}} \}  = 
(\tilde{r}/r) \cdot \tilde{\partial} \{ \psi_{L}(f_1^{(\alpha)}), \cdots, \psi_{L}(f_q^{(\alpha)}) \} 
\end{equation*}
in $K^{M}_{q-1}{\C}(\tilde{Z})$ by the formula \eqref{eqn: tame symbol} for tame symbol.  
Hence we have 
\begin{eqnarray*}
\psi_{{\G}'}({\Res}(\omega)) 
& = & 
(\tilde{r}/r) \cdot \tilde{\partial} \left(  \sum_{\alpha} \{ \psi_{L}(f_1^{(\alpha)}), \cdots, \psi_{L}(f_q^{(\alpha)}) \} \right) \\ 
& = & 
(\tilde{r}/r) \cdot \tilde{\partial}(\psi_{L}(\omega)) 
\end{eqnarray*}
in $K^{M}_{q-1}{\C}(\tilde{Z})$. 
By the compatibility of tame symbol and pullback (\cite{GS} Remark 7.1.6.2), 
we have 
\begin{equation*}
(\tilde{r}/r) \cdot \tilde{\partial}(\psi_{L}(\omega)) = \partial(\psi_{{\G}}(\omega))  
\end{equation*}
in $K^M_{q-1}{\C}(X')\subset K^M_{q-1}{\C}(\tilde{Z})$. 
This proves our assertion.  
\end{proof}

\begin{remark}\label{rmk: two nu}
To make the diagram \eqref{diagram; residue} commutative is 
the reason of our modification \eqref{eqn: nu} in the definition of $\nu$. 
If we use $\tilde{\nu}$, it does not depend on ${\G}$, 
but instead the diagram \eqref{diagram; residue} commutes only up to constant in general. 
\end{remark}

\subsection{Residue maps and transfers}\label{ssec: residue transfer}

In this subsection we prove 
compatibility of residue maps and transfer maps in the modular side. 
We begin by preparing notations. 
Let $\tilde{{\G}}$ be a subgroup of ${\G}$ containing ${\Ohat}(L)$. 
We put $\tilde{G}=\tilde{{\G}}/{\Ohat}(L)$. 
We choose representatives $\{ L_{\alpha}' \}_{\alpha}$ of $\tilde{{\G}}$-equivalence classes of 
sublattices of $L$ which are ${\G}$-equivalent to $L'$. 
We write ${\G}_{\alpha}'$ 
for the image of ${\rm Stab}_{\tilde{{\G}}}(L_{\alpha}') \to {\rm O}^{+}(L_{\alpha}')$. 
Under the identification $L_{\alpha}'\simeq L'$, 
we have 
\begin{equation*}
{\Ohat}(L') < {\G}_{\alpha}' < {\G}' < {\rm O}^{+}(L'). 
\end{equation*}
We put $G_{\alpha}'={\G}_{\alpha}'/{\Ohat}(L')$. 
The corresponding geometric situation is 
\begin{equation*}\label{eqn: CD transfer geom}
\xymatrix{
\bigsqcup_{\alpha} X_{{\G}_{\alpha}'} \ar[r] \ar[d] & \sum_{\alpha} Z_{\alpha}  \ar@{^{(}->}[r] \ar[d] & \tilde{X} \ar[d] \\ 
X'  \ar[r] & Z \ar@{^{(}->}[r] & X 
}
\end{equation*}
where $\tilde{X}=X_{\tilde{{\G}}}$, 
$Z_{\alpha}\subset \tilde{X}$ is the special divisor for $L_{\alpha}'\subset L$, 
the left horizontal maps are normalization maps, 
and $\sum_{\alpha} Z_{\alpha}$ is the inverse image of $Z$ in $\tilde{X}$. 

\begin{proposition}\label{prop: residue and transfer}
The following diagram commutes: 
\begin{equation*}
\xymatrix@C+2pc{
\wedge^{q}M(L)^{\tilde{G}} \ar[r]^-{({\rm Res}_{\alpha})_{\alpha}} \ar[d]_{N} & 
\bigoplus_{\alpha} \wedge^{q-1}M(L_{\alpha}')^{G_{\alpha}'} \ar[d]^{(N_{\alpha})_{\alpha}} \\ 
\wedge^{q}M(L)^{G} \ar[r]^{{\rm Res}} & \wedge^{q-1}M(L')^{G'} 
}
\end{equation*}
Here ${\rm Res}$ is the residue map for $({\G}; {\G}')$, 
${\rm Res}_{\alpha}$ is the residue map for $(\tilde{{\G}}; {\G}_{\alpha}')$, 
$N$ is the transfer map for $\tilde{{\G}}<{\G}$, and 
$N_{\alpha}$ is the transfer map for ${\G}_{\alpha}'<{\G}'$. 
\end{proposition}

\begin{proof}
By the transitivity and surjectivity of transfer maps, we are reduced to the case 
$\tilde{{\G}} = {\Ohat}(L)$. 
Then $\tilde{G}$ is trivial. 
Since ${\Ohat}(L)$ is normal in ${\G}$, 
each $G_{\alpha}'$ is normal in $G'$ with $G'/G_{\alpha}'$ the Galois group of 
$X_{{\G}_{\alpha}'}\to X'$. 
Let $H_{\alpha}<G$ be the stabilizer of $Z_{\alpha} \leftrightarrow [L_{\alpha}']$ 
and $R_{\alpha}<H_{\alpha}$ be the inertia group. 
Then $|R_{\alpha}|$ is the ramification index of $\tilde{X}_{L}\to X$ along $Z_{\alpha}\to Z$, 
and the quotient $H_{\alpha}/R_{\alpha}$ is identified with $G'/G_{\alpha}'$. 
We have 
$G=\bigsqcup_{\alpha} \gamma_{\alpha} H_{\alpha}$ 
where $\gamma_{\alpha}$ is represented by an element $\gamma_{\alpha}$ of ${\G}$ 
with $\gamma_{\alpha}L_{\alpha}'=L'$ 
(we use the same notation $\gamma_{\alpha}$). 

We take an element $f_1\wedge \cdots \wedge f_q$ of $\wedge^{q}M(L)$. 
Then we have 
\begin{eqnarray*}
& & {\Res} \circ N (f_1\wedge \cdots \wedge f_q) 
 =  
\sum_{\gamma \in G} {\Res} (\gamma f_1\wedge \cdots \wedge \gamma f_q ) \\ 
& = & 
\sum_{\alpha} \sum_{\gamma\in \gamma_{\alpha}H_{\alpha}} \sum_{i=1}^{q} 
(-1)^{i-1} \nu_{L'}(\gamma f_i) \cdot 
(\gamma f_1|_{L'}) \wedge \cdots ({\rm no} \; i) \cdots \wedge (\gamma f_q|_{L'}) \\ 
& = & 
\sum_{\alpha} \sum_{\gamma\in H_{\alpha}} \sum_{i=1}^{q} 
(-1)^{i-1} \nu_{L'}(\gamma_{\alpha} \gamma f_i) \cdot 
(\gamma_{\alpha} \gamma f_1|_{L'}) \wedge \cdots ({\rm no} \; i) \cdots \wedge (\gamma_{\alpha} \gamma f_q|_{L'}) 
\end{eqnarray*}
in $\wedge^{q-1}M(L')$. 
We have 
$(\gamma_{\alpha} \gamma f)|_{L'} = \gamma (f|_{L_{\alpha}'})$ 
under the identification $\gamma_{\alpha} \colon L_{\alpha}'\to L'$. 
Moreover, since $|R_{\alpha}|\cdot \nu_{L'}=\nu_{L_{\alpha}'}$, 
we have 
\begin{equation*}
|R_{\alpha}| \cdot \nu_{L'}(\gamma_{\alpha} \gamma f_i) = 
\nu_{L_{\alpha}'}(\gamma f_i) = \nu_{\gamma^{-1}L_{\alpha}'}(f_i) = \nu_{L_{\alpha}'}(f_i). 
\end{equation*}
Therefore, via $H_{\alpha}/R_{\alpha}\simeq G'/G_{\alpha}'$,  
the above expression is equal to 
\begin{eqnarray*}
& & 
\sum_{\alpha} \sum_{\gamma\in G'/G_{\alpha}'} \sum_{i=1}^{q} 
(-1)^{i-1} \nu_{L_{\alpha}'}(f_i) \cdot 
\gamma (f_1|_{L_{\alpha}'}) \wedge \cdots ({\rm no} \; i) \cdots \wedge \gamma (f_q|_{L_{\alpha}'}) \\ 
& = &  
\sum_{\alpha} \sum_{\gamma\in G'/G_{\alpha}'} 
\gamma ({\Res}_{\alpha}(f_1 \wedge \cdots \wedge f_q)) \\ 
& = & 
\sum_{\alpha}N_{\alpha} \circ {\Res}_{\alpha} (f_1\wedge \cdots \wedge f_q). 
\end{eqnarray*}
This proves our assertion. 
\end{proof}

\section{Borcherds-Gersten complex}\label{sec: BG cpx}

We now come to the main point of this paper. 
We put the constructions in the previous sections together 
to define a modular analogue of the Gersten complex. 
The proof of the required properties will be given in the next \S \ref{ssec: complex}. 

We fix an even lattice $L_0$ of signature $(2, n)$ with $n>0$ 
and a subgroup ${\G}$ of ${\rm O}^{+}(L_0)$ containing ${\Ohat}(L_0)$. 
The modular variety $X=X_{{\G}}$ is the target variety 
whose Gersten complex we want to approximate by Borcherds products. 
The notations $L, L'$ will be used for primitive sublattices of $L_0$ of signature $(2, \ast)$. 
As in \S \ref{ssec: special cycle}, 
we denote by ${\GL}$ the image of ${\rm Stab}_{{\G}}(L)\to {\OL}$. 
We also write $X_L=X_{{\GL}}$ and $G_L={\GL}/{\Ohat}(L)$. 
Recall that $X_L$ is the normalization of the special cycle $Z_L\subset X$ for $L\subset L_0$. 
On the other hand, as in \eqref{eqn: tildeXL}, 
we write $\tilde{X}_{L}=X_{{\Ohat}(L)}$. 
Then $X_L=\tilde{X}_L/G_L$. 
These notations are fixed throughout this section.

\subsection{Boundary map}\label{ssec: boundary map}

Let $L'\subset L$ be sublattices of $L_0$ of corank $r+1$ and $r$ respectively. 
In this subsection, we define the boundary map 
\begin{equation*}
\partial^{L}_{L'} : \wedge^{q}M(L)^{G_L} \to \wedge^{q-1}M(L')^{G_{L'}} 
\end{equation*}
by adjusting the residue map in \S \ref{sec: residue} to the relative setting $L'\subset L\subset L_0$. 
Since what complicates the construction is the possible non-normality of special cycles, 
let us begin by explaining the relevant geometric situation. 

We denote by 
$\mathcal{C}(L', L; {\G})$ 
the set of ${\GL}$-equivalence classes of sublattices of $L$ which are ${\G}$-equivalent to $L'$. 
We choose a representative $\{ L_{\alpha}' \}_{\alpha}$ of $\mathcal{C}(L', L; {\G})$. 
If $\pi_{L}\colon X_L\to Z_L$ is the normalization map, we have 
\begin{equation*}
\pi_{L}^{-1}(Z_{L'}) = \sum_{\alpha} Z_{L_{\alpha}'/L}, 
\end{equation*}
where 
$Z_{L_{\alpha}'/L}\subset X_{L}$ is the special cycle for $L_{\alpha}'\subset L$. 
We denote by ${\G}_{L_{\alpha}'/L}$ 
the image of ${\rm Stab}_{{\GL}}(L_{\alpha}')\to {\rm O}^{+}(L_{\alpha}')$. 
Then $X_{L_{\alpha}'/L}:=X_{{\G}_{L_{\alpha}'/L}}$ is the normalization of $Z_{L_{\alpha}'/L}$. 
If we view ${\G}_{L_{\alpha}'/L}$ as a subgroup of ${\rm O}^{+}(L')$ 
by the ${\G}$-action $L_{\alpha}'\simeq L'$, 
we have 
\begin{equation}\label{eqn: groups 1}
{\Ohat}(L') \: \subset \: {\G}_{L_{\alpha}'/L}  \: \subset \: {\G}_{L'}  \: \subset \: {\rm O}^{+}(L'). 
\end{equation}
We write $G_{\alpha} = {\G}_{L_{\alpha}'/L}/{\Ohat}(L')$. 
For each $\alpha$, the situation is summarized in the commutative diagram 
\begin{equation*}
\xymatrix@C+1pc{
\tilde{X}_{L'} \ar[rr] \ar[d]^{/G_{\alpha}} \ar@/_24pt/[dd]_{/G_{L'}} & & \tilde{X}_{L} \ar[d]^{/G_{L}} & \\ 
X_{L_{\alpha}'/L} \ar[r]^{\pi_{\alpha}}  \ar[d] & Z_{L_{\alpha}'/L} \ar[d] \ar@{^{(}->}[r] & X_L \ar[d]^{\pi_{L}} & \\ 
 X_{L'} \ar[r]^{\pi_{L'}}  & Z_{L'} \ar@{^{(}->}[r] & Z_{L} \ar@{^{(}->}[r] & X
}
\end{equation*}
where various $\pi_{?}$ are normalization maps and 
the left column comes from \eqref{eqn: groups 1}. 

Now we define the boundary map. 
The modular residue map for 
the sublattice $L_{\alpha}'\subset L$ with the groups $({\G}_{L}; {\G}_{L_{\alpha}'/L})$ 
is a map of the form 
\begin{equation*}
{\Res}_{\alpha}:= {\Res}^{L}_{L_{\alpha}'} \; : \; 
\wedge^{q}M(L)^{G_L} \to \wedge^{q-1}M(L')^{G_{\alpha}}. 
\end{equation*}
We compose it with the transfer map 
\begin{equation*}
N_{\alpha} : \wedge^{q-1}M(L')^{G_{\alpha}} \to \wedge^{q-1}M(L')^{G_{L'}}
\end{equation*}
for ${\G}_{L_{\alpha}'/L} \subset {\G}_{L'}$. 
Finally, we take the sum over $\alpha$: 
\begin{equation*}
\partial^{L}_{L'} := \sum_{\alpha} N_{\alpha} \circ {\Res}_{\alpha} : 
\wedge^{q}M(L)^{G_L} \to \wedge^{q-1}M(L')^{G_{L'}}. 
\end{equation*}
This is our boundary map.

\subsection{The complex}\label{ssec: BG cpx}

Now we can define our complex. 
Let $p\leq n$. 
For $r\leq p$, let $L_{0}^{(r)}$ be the set of ${\G}$-equivalence classes of sublattices of $L_0$ of corank $r$. 
This is the same as the set of special cycles of $X$ of codimension $r$. 
Then we form the sequence 
\begin{equation*}
{\Mp} \; : \; \cdots \longrightarrow \bigoplus_{L\in L_0^{(r)}}\wedge^{p-r}M(L)^{G_L} 
\stackrel{\partial}{\longrightarrow} \bigoplus_{L'\in L_0^{(r+1)}}\wedge^{p-r-1}M(L')^{G_{L'}} 
\longrightarrow \cdots 
\end{equation*}
where the term indexed by $L_0^{(r)}$ is placed in degree $r$, 
and $\partial=(\partial^{L}_{L'})$ consists of the boundary maps defined in \S \ref{ssec: boundary map}.  
(We put $\partial^{L}_{L'}=0$ if $Z_{L'}\not\subset Z_{L}$.)  
Thus ${\Mp}$ has length $p+1$, 
starting from $\wedge^{p}M(L_0)^{{\G}}$ in degree $0$: 
\begin{equation*}
\wedge^{p}M(L_0)^{{\G}} \longrightarrow \cdots \longrightarrow 
\bigoplus_{L\in L_{0}^{(p-1)}} M(L)^{G_{L}} \longrightarrow 
\bigoplus_{L'\in L_{0}^{(p)}} {\Q}[L']. 
\end{equation*}

\begin{example}
As the simplest case, suppose that $L_0$ is unimodular. 
Then ${\G}={\rm O}^{+}(L_0)$. 
By the theory of Nikulin (\cite{Ni} \S 1.14), 
every negative-definite even lattice $K$ of rank $r<n/2$ can be primitively embedded in $L_0$, 
uniquely up to ${\rm O}^{+}(L_0)$. 
Therefore, by the correspondence $K\leftrightarrow K^{\perp}$, 
$L_{0}^{(r)}$ can be identified with the set $\mathcal{S}_{r}$ 
of isometry classes of negative-definite even lattices of rank $r$. 
Furthermore, 
$G_{K^{\perp}}$ is the image of ${\rm O}(K)\to {\rm O}(A_K)\simeq {\rm O}(A_{K^{\perp}})$ 
divided by $\pm 1$. 
Thus, in the range $r<n/2$, 
${\Mp}$ can be expressed as 
\begin{equation*}
\wedge^{p}M(L_0) \to \bigoplus_{d\in \mathbb{N}} \wedge^{p-1}M(\langle -2d \rangle^{\perp}) \to 
\cdots \to \bigoplus_{K\in \mathcal{S}_{r}} \wedge^{p-r}M(K^{\perp})^{{\rm O}(K)} \to \cdots 
\end{equation*}
Since $A_{K^{\perp}}\simeq A_{K}(-1)$, 
each $M(K^{\perp})$ can be written purely in terms of $K$. 
\end{example}

%
%
We go back to the general situation. 
For $L\in L_{0}^{(r)}$ we denote by 
\begin{equation*}
\psi_{L} = \psi_{{\GL}} \: : \: \wedge^{q}M(L)^{G_L} \to K^{M}_{q}{\C}(X_L) 
\end{equation*}
the Borcherds lift map for ${\GL}$. 
(The abbreviation $\psi_{L}$ is different from that in \S \ref{sec: Borcherds to Milnor} and \S \ref{sec: residue}.) 
Our main result is the following. 

\begin{theorem}\label{thm: main}
Let $p\leq n$. 
The following holds. 

(1) The sequence ${\Mp}$ is a complex, i.e., $\partial \circ \partial =0$. 

(2) The Borcherds lift maps $\psi_{L}$ for sublattices $L$ of $L_0$ define a morphism 
${\Mp} \to {\Kp}$
of complexes. 
\end{theorem} 

The proof of Theorem \ref{thm: main} will be given in \S \ref{ssec: complex}. 
Before going ahead, 
let us look at its relevance to the higher Chow groups in the rest of this section. 

\begin{remark}\label{rmk: p>n}
We put the assumption $p\leq n$ so that 
$\bigoplus_{L_{0}^{(n)}}{\Q}[L]$ is the only place where special $0$-cycles appear. 
Since we do not have (arithmetic normalization of) Borcherds products in dimension $0$, 
the Borcherds lift map is not well-defined for special $0$-cycles. 
For this reason, in the case $p>n$, we need to truncate ${\Mp}$ at special $0$-cycles. 
\end{remark}

\subsection{Special higher cycles}\label{ssec: special higher cycle}

We recall from \S \ref{ssec: higher Chow} the relation of the Gersten complexes and the higher Chow groups. 
Combining the morphism ${\Mp}\to {\Kp}$ with the isomorphism \eqref{eqn: Gersten higher Chow}, 
we obtain the ${\Q}$-linear map 
\begin{equation}\label{eqn: modular to higher Chow}
H^{p-m}{\Mp} \to {\CH}^{p}(X, m)_{\Q}, \quad 0\leq m \leq 2. 
\end{equation}
We call higher Chow cycles obtained in this way \textit{special higher cycles}. 
When $m\geq 3$, the relation is more indirect: 
\begin{equation*}
H^{p-m}{\Mp} \to H^{p-m}{\Kp} \leftarrow {\CH}^{p}(X, m)_{\Q}, 
\end{equation*}
where the right map is the edge map \eqref{eqn: CH coniveau} in the coniveau spectral sequence. 
As explained in \eqref{eqn: Chow Gersten p=3}, this is surjective when $m=3$, 
so in this case we still obtain higher Chow cycles from $H^{p-3}\mathsf{M}^{\bullet}_{p,{\G}}$. 
In the case $m\geq 4$, it would be a natural problem to give a criterion, 
in terms of modular forms, 
for the Gersten cohomology classes coming from ${\Mp}$ 
to lie in the image of ${\CH}^{p}(X, m)_{\Q}$, 
or in other words, to survive until the $E_{\infty}$-page of the coniveau spectral sequence. 

In the following we shall be more specific for each $0\leq m \leq 2$.

\subsubsection{The case $m=0$}

The map 
$\mathsf{M}_{p,{\G}}^{p} \to {\CH}^p(X)_{\Q}$ 
is just the map sending a sublattice $L$ to the class of the special cycle $Z_{L}$. 
The image of $\mathsf{M}_{p,{\G}}^{p-1} \to \mathsf{M}_{p,{\G}}^{p}$ 
is the space of rational equivalences provided by 
Borcherds products on special cycles of one larger dimension.  
In the case $p=1$, a theorem of Bruinier \cite{Br} and Bergeron-Li-Millson-Moeglin \cite{BLMM} says that 
$H^1\mathsf{M}_{1,{\G}}^{\bullet} \to {\rm Pic}(X)_{\Q}$ 
is an isomorphism under suitable condition on $L_{0}$. 
Moreover, a theorem of Bruinier-Raum-Zhang (\cite{BR}, \cite{Zhang}) says that 
the image of $H^p{\Mp}\to {\CH}^p(X)_{\Q}$ is finite-dimensional for every $p$. 
It is plausible that $H^p{\Mp}$ itself has finite dimension too. 


\subsubsection{The case $m=1$}

Special $(p, 1)$-cycles are of the form 
\begin{equation*}
\sum_{i}(Z_{L_i}, \psi(f_i)), 
\end{equation*}
where $Z_{L_i}$ is a special cycle of codimension $p-1$ 
and $\psi(f_i)$ is a Borcherds product on $Z_{L_i}$. 
The cocycle condition is 
$\sum_{i}\nu_{L'}(f_i)=0$ 
for any $L'\subset L_0$ of corank $p$. 
This can be expressed by the Fourier coefficients of $f_i$. 
In the case $p=2$, such cycles were considered by Garcia \cite{Ga} from the viewpoint of theta lifting. 

A $(p, 1)$-cycle is called \textit{decomposable} if it is rationally equivalent to  
a sum of cycles of the form $(Z, \psi)$ 
where $\psi\in \mathcal{O}^{\times}(Z)$ has no zero nor pole. 
(In modular situation, $\psi$ must be constant in most cases by the Koecher principle.) 
The quotient of ${\CH}^p(X, 1)_{\Q}$ by the subspace of decomposable cycles is denoted by 
${\CH}^p(X, 1)_{{\rm ind}, {\Q}}$. 
Since Borcherds products rarely have trivial divisor, 
it will be more natural to pass to the indecomposable part: 
\begin{equation}\label{eqn: indecomposable}
H^{p-1}{\Mp} \to {\CH}^p(X, 1)_{{\rm ind}, {\Q}}. 
\end{equation}
Then $H^{p-1}{\Mp}$ has countable dimension, 
while ${\CH}^p(X, 1)_{{\rm ind}, {\Q}}$ 
is conjectured to have countable dimension at least for $p=2$ and $X$ smooth projective. 
Hence it would be natural to ask 
whether the kernel and cokernel of \eqref{eqn: indecomposable} have finite dimension. 

As an illustration of our construction, we give an example of special $(2, 1)$-cycle. 
(I suspect, however, that this cycle would vanish in $H^{1}\mathsf{M}_{2,{\G}}^{\bullet}$.) 

\begin{example}
We take $L_0=2U\oplus A_1^{(1)}\oplus A_1^{(2)}$ 
where $A_1^{(i)}$ is a copy of $A_1$. 
Let ${\G}={\Ohat}(L_0)$. 
We set $L_i=2U\oplus A_1^{(i)}$ for $i=1, 2$. 
We also choose a $(-2)$-vector $\delta$ from $2U$ and put 
\begin{equation*}
L_3 = \delta^{\perp}\cap 2U \simeq U\oplus \langle 2 \rangle \oplus 2A_1 \simeq U\oplus U(2) \oplus A_1. 
\end{equation*}
We write $L_{ij}=L_i \cap L_j$ for $i\ne j$. 
Then $L_{12}=2U$ and $L_{i3}\simeq U\oplus \langle 2 \rangle \oplus A_1$ for $i=1, 2$. 
It is easy to see that 
$Z_{L_{i}}\cap Z_{L_{j}} = Z_{L_{ij}}$. 

Now $X_{L_1}\simeq X_{L_{2}}$ are the Siegel modular threefold for ${\rm Sp}_{4}({\Z})$. 
We take the modular function $\Delta_{5}^{7}/\chi_{35}$, 
where $\Delta_{5}$ is the product of the ten even theta constants 
and $\chi_{35}$ is the first Siegel modular form of odd weight. 
(The subscript indicates the weight.) 
By Gritsenko-Nikulin \cite{GN}, this is a Borcherds product with 
\begin{equation*}
{\rm div}(\Delta_{5}^{7}/\chi_{35}) = 7 {\rm div}(\Delta_{5}) - {\rm div}(\chi_{35}) = 6 Z_{L_{12}} - Z_{L_{i3}}. 
\end{equation*}
On the other hand, $L_3$ is a simple lattice in the sense of \cite{BEF}. 
Hence there is a Borcherds product $\psi_1$ on $X_{L_3}$ with ${\rm div}(\psi_1)=Z_{L_{13}}$. 
By applying the involution of $L_{3}$ exchanging the two copies of $A_1$, 
we also obtain a Borcherds product $\psi_2$ of the same weight with ${\rm div}(\psi_2)=Z_{L_{23}}$. 
Then $\psi_1/\psi_2$ has weight $0$ and ${\rm div}(\psi_3)=Z_{L_{13}}-Z_{L_{23}}$. 
Thus the element 
\begin{equation*}
(L_1, \Delta_{5}^{7}/\chi_{35}) + (L_2, \chi_{35}/\Delta_{5}^{7}) + (L_3, \psi_1/\psi_2) 
\end{equation*}
of $\mathsf{M}_{2,{\G}}^{1}$ satisfies the cocycle condition. 
\end{example}

\subsubsection{The case $m=2$}

Special $(p, 2)$-cycles are of the form 
\begin{equation*}
\sum_{i}(Z_{L_i}, \{ \psi(f_i), \psi(g_i) \}), 
\end{equation*}
where $Z_{L_i}$ is a special cycle of codimension $p-2$ and 
$\psi(f_i), \psi(g_i)$ are Borcherds products on $Z_{L_i}$. 
The cocycle condition is 
\begin{equation*}
\sum_{i}\partial^{L_i}_{L'}(f_i\wedge g_i) = 0 
\end{equation*}
for any $L'\subset L_0$ of corank $p-1$. 

For example, when $p=2$, 
we have $L_i=L_{0}$ and $\partial^{L_{0}}_{L'}={\Res}^{L_{0}}_{L'}$ is the simple residue map. 
Then, by \eqref{eqn: residue q=2}, 
the cocycle condition can be written as 
\begin{equation}\label{eqn: cocyle condition (2, 2)}
\left. \left( \sum_{i} \nu_{L'}(f_i) g_i - \nu_{L'}(g_i) f_i \right) \; \right|_{L'} = 0. 
\end{equation}
Sreekantan (\cite{Sr2} \S 6.0.2) asked whether it is possible to express 
the cocycle condition for special $(2, 2)$-cycles in terms of input modular forms. 
The condition \eqref{eqn: cocyle condition (2, 2)} provides an answer in the case $n\geq 2$. 
His question, however, was originally asked when $n=1$, and this case is not covered in Theorem \ref{thm: main} 
as explained in Remark \ref{rmk: p>n}. 
What enters here is CM values of Borcherds products on modular curves.  
This remaining case will be interesting from both $K$-theoretic and modular viewpoints.


\section{Proof of Theorem \ref{thm: main}}\label{ssec: complex}

In this section we prove Theorem \ref{thm: main}. 
We first prove the assertion (2); 
then we prove the assertion (1) by utilizing the Borcherds lift maps. 
We keep the notation in \S \ref{sec: BG cpx}.

\subsection{Borcherds lift map}\label{ssec: Borcherds map Gersten}

Theorem \ref{thm: main} (2) amounts to the following assertion. 

\begin{proposition}\label{prop: BG to G}
Let $L'\subset L \subset L_0$ be as in \S \ref{ssec: boundary map}. 

(1) The following diagram commutes: 
\begin{equation*}
\xymatrix@C+1pc{
\wedge^{q}M(L)^{G_{L}} \ar[r]^{\partial^{L}_{L'}} \ar[d]_{\psi_{L}} & \wedge^{q-1}M(L')^{G_{L'}} \ar[d]^{\psi_{L'}}   \\ 
K^{M}_{q} {\C}(X_{L}) \ar[r]^{\partial^{Z}_{Z'}}  & K^{M}_{q-1}{\C}(X_{L'}) 
}
\end{equation*}
where $\partial^{Z}_{Z'}$ is the $K$-theoretic boundary map for $Z_{L'}\subset Z_{L}$. 

(2) If $Y$ is an irreducible divisor of $X_L$ which is not a special divisor, 
the composition 
\begin{equation*}
\wedge^{q}M(L)^{G_L} \stackrel{\psi_L}{\longrightarrow} K^{M}_{q}{\C}(X_L) 
\stackrel{\partial_{Y}}{\longrightarrow} K^{M}_{q-1}{\C}(Y) 
\end{equation*}
is zero, where $\partial_{Y}$ is the tame symbol for $Y\subset X_L$.  
\end{proposition}

\begin{proof}
(1) By the definition of $\partial^{L}_{L'}$ and $\partial^{Z}_{Z'}$ 
and the geometric explanation in \S \ref{ssec: boundary map}, 
the diagram in question can be factorized as 
\begin{equation*}
\xymatrix@C+1pc{
\wedge^{q}M(L)^{G_{L}} \ar[r]^-{({\Res}^{M}_{\alpha})} \ar[d]_{\psi_{L}} & 
\bigoplus_{\alpha}\wedge^{q-1}M(L_{\alpha}')^{G_{\alpha}} \ar[d]^{(\psi_{\alpha})} \ar[r]^-{(N_{\alpha}^{M})}  & 
\wedge^{q-1}M(L')^{G_{L'}} \ar[d]^{\psi_{L'}}  \\ 
K^{M}_{q} {\C}(X_{L}) \ar[r]^-{({\Res}_{\alpha}^{K})}  & 
\bigoplus_{\alpha}K^{M}_{q-1} {\C}(X_{L_{\alpha}'/L}) \ar[r]^-{(N_{\alpha}^{K})}  & 
K^{M}_{q-1}{\C}(X_{L'}) 
}
\end{equation*}
Here ${\Res}^{M}_{\alpha}$ is the modular residue map for $({\G}_{L}; {\G}_{L_{\alpha}'/L})$, 
${\Res}^{K}_{\alpha}$ is the tame symbol for $Z_{L_{\alpha}'/L}\subset X_{L}$, 
$N_{\alpha}^{M}$ is the modular transfer map for ${\G}_{L_{\alpha}'/L}\subset {\G}_{L'}$, 
$N_{\alpha}^{K}$ is the $K$-theoretic transfer map for $X_{L_{\alpha}'/L}\to X_{L'}$, and 
$\psi_{\alpha}$ is the Borcherds lift map for ${\G}_{L_{\alpha}'/L}$. 
Then the left square commutes by Proposition \ref{prop: residue and tame}, 
and the right square commutes by Lemma \ref{lem: push}. 

(2) We choose an irreducible divisor $\tilde{Y}$ of $\tilde{X}_{L}$ lying over $Y\subset X_L$ 
and let $\partial_{\tilde{Y}}$ be the tame symbol for $\tilde{Y} \subset \tilde{X}_{L}$. 
We consider the diagram 
\begin{equation*}
\xymatrix@C+1pc{
\wedge^{q}M(L)^{G_{L}} \ar[r]^-{\psi_L} \ar@{^{(}->}[d] & 
K^{M}_{q}{\C}(X_L) \ar[r]^-{\partial_{Y}} \ar@{^{(}->}[d] & 
K^{M}_{q-1}{\C}(Y) \ar@{^{(}->}[d]  \\ 
\wedge^{q}M(L) \ar[r]^-{\psi_L}  & 
K^{M}_{q}{\C}(\tilde{X}_L) \ar[r]^-{\partial_{\tilde{Y}}} & 
K^{M}_{q-1}{\C}(\tilde{Y})
} 
\end{equation*}
The left square commutes by Lemma \ref{lem: pullback}, 
and the right square commutes up to constant 
by the functoriality of tame symbol (\cite{GS} Remark 7.1.6.2). 
Hence it suffices to show that the composition in the lower sequence is zero. 

We take an element $f_1\wedge \cdots \wedge f_q$ of $\wedge^{q}M(L)$. 
Since ${\rm div}(\psi_{L}(f_i))$ is a linear combination of special divisors, 
it does not contain $\tilde{Y}$ in its support. 
Hence 
\begin{equation*}
\partial_{\tilde{Y}} (\psi_{L}(f_1\wedge \cdots \wedge f_q)) = 
\partial_{\tilde{Y}} \{ \psi_{L}(f_1), \cdots, \psi_{L}(f_q) \} = 0. 
\end{equation*}
This proves our assertion. 
\end{proof}

\subsection{Proof of $\partial \circ \partial =0$}

Theorem \ref{thm: main} (1) amounts to the following assertion. 

\begin{proposition}\label{prop: complex}
Let $L''\subset L$ be sublattices of $L_0$ of corank $r+2$ and $r$ respectively. 
Then the composition 
\begin{equation}\label{eqn: localized complex}
\wedge^{q}M(L)^{G_L} \xrightarrow{(\partial^{L}_{L'})}
\bigoplus_{L'} \wedge^{q-1}M(L')^{G_{L'}} \xrightarrow{(\partial^{L'}_{L''})} 
\wedge^{q-2}M(L'')^{G_{L''}} 
\end{equation}
is zero, 
where $L'$ ranges over $L'\in L_0^{(r+1)}$ with 
$Z_{L''}\subset Z_{L'} \subset Z_{L}$. 
\end{proposition}

\begin{proof}
The proof is a succession of reduction, 
eventually reduced to the case $q=2$. 
Since the reduction process is rather long, 
we present the proof in the reverse order for the sake of readability.  
Thus we begin with the case $q=2$, 
gradually extend it, and finally arrive at the general case. 

\begin{step}
Proposition \ref{prop: complex} holds in the case $q=2$. 
\end{step}

\begin{proof}
By Proposition \ref{prop: BG to G}, we have the commutative diagram 
\begin{equation*}
\xymatrix@C+1pc{
\wedge^{2}M(L)^{G_{L}} \ar[r] \ar[d]_{\psi_{L}} & 
\bigoplus_{L'} M(L')^{G_{L'}} \ar[d]^{(\psi_{L'})} \ar[r] & {\Q}[L''] \ar[d]^{=} \\ 
K_{2} {\C}(X_{L}) \ar[r] & 
\bigoplus_{Z'}{\C}(Z')^{\times} \ar[r]  & {\Q}[Z_{L''}]
}
\end{equation*}
The lower sequence is a part of the Gersten complex for $X$, 
where $Z'$ ranges over $Z'\in X^{(r+1)}$ with $Z_{L''}\subset Z' \subset Z_{L}$. 
Since the right vertical map is an isomorphism, 
the fact that the lower sequence is a complex (\cite{Ka}) implies that 
the upper sequence is so too. 
\end{proof}

Next we consider the case $r=0$ with $G_L$ trivial. 

\begin{step}
Proposition \ref{prop: complex} holds in the case $L=L_0$, ${\G}={\Ohat}(L)$. 
\end{step}

\begin{proof}
By our assumption, the sequence \eqref{eqn: localized complex} takes the form 
\begin{equation*}
\wedge^{q}M(L) \stackrel{({\Res}^{L}_{L'})}{\longrightarrow} 
\bigoplus_{L'} \wedge^{q-1}M(L')^{G_{L'}} \stackrel{(\partial^{L'}_{L''})}{\longrightarrow} 
\wedge^{q-2}M(L'')^{G_{L''}}. 
\end{equation*}
%
We first calculate $\partial_{L''}^{L'} \circ {\Res}^{L}_{L'}$ for each $L'$ with $L''\subset L' \subset L$. 
We begin with ${\Res}^{L'}_{L''}\circ {\Res}^{L}_{L'}$. 
We take an element 
$\omega = f_1\wedge \cdots \wedge f_q$ of $\wedge^q M(L)$. 
Then 
\begin{eqnarray*}
& & {\Res}^{L'}_{L''} ({\Res}^{L}_{L'}(\omega))  \\ 
& = & 
{\Res}^{L'}_{L''} \left( \sum_{i=1}^{q} (-1)^{i-1}\nu_{L'}(f_i)\cdot (f_1|_{L'}) \wedge \cdots (\textrm{no} \; i) \cdots \wedge (f_q|_{L'}) \right) \\ 
& = & 
\sum_{i=1}^{q} \sum_{j=1}^{q-1} (-1)^{i-1}(-1)^{j-1} \nu_{L'}(f_i) \cdot \nu_{L''}(f_{\tilde{j}}|_{L'}) \cdot 
(f_1|_{L'})|_{L''} \wedge \cdots (\textrm{no} \; i, \tilde{j}) \cdots \wedge (f_q|_{L'})|_{L''} 
\end{eqnarray*}
where $\tilde{j}=j$ if $j<i$ and $\tilde{j}=j+1$ if $j\geq i$. 
By the transitivity of quasi-pullback, we have 
$(f|_{L'})|_{L''} = f|_{L''}$. 
Hence, rewriting $\tilde{j}$ as $j$ and putting 
\begin{equation*}
\omega_{i,j} = f_1|_{L''} \wedge \cdots (\textrm{no} \; i, j) \cdots \wedge f_q|_{L''}, 
\end{equation*} 
we obtain  
\begin{equation*}
{\Res}^{L'}_{L''} ({\Res}^{L}_{L'}(\omega)) = 
\sum_{i\ne j} (-1)^{i+j+\varepsilon(i, j)} \nu_{L'}(f_i) \cdot \nu_{L''}(f_{j}|_{L'}) \cdot \omega_{i,j}, 
\end{equation*}
where $\varepsilon(i, j)=0$ if $j<i$ and $\varepsilon(i, j)=1$ if $j>i$. 
It is convenient to rewrite this as 
\begin{equation*}
{\Res}^{L'}_{L''} ({\Res}^{L}_{L'}(\omega)) =  
\sum_{i>j} (-1)^{i+j} \{ \nu_{L'}(f_i) \nu_{L''}(f_{j}|_{L'}) - \nu_{L'}(f_j) \nu_{L''}(f_{i}|_{L'}) \} \cdot \omega_{i,j}.  
\end{equation*}

The next step is to apply the transfer map of ${\G}_{L''/L'}<{\G}_{L''}$. 
If $f\in M(L)$, $f|_{L''}$ is not just ${\G}_{L''/L'}$-invariant but already ${\G}_{L''}$-invariant. 
By the above expression, we find that 
${\Res}^{L'}_{L''} ({\Res}^{L}_{L'}(\omega))$ is already ${\G}_{L''}$-invariant. 
Hence the transfer map for ${\Res}^{L'}_{L''} ({\Res}^{L}_{L'}(\omega))$ 
is just the multiplication by $[{\G}_{L''}:{\G}_{L''/L'}]$. 

The final step is to take the sum over 
$\mathcal{C}(L'', L'; {\G})=\{ L_{\alpha}'' \}$. 
(See \S \ref{ssec: boundary map} for the notation.) 
For each $\alpha$, 
we have $f|_{L_{\alpha}''} = f|_{L''}$ by the ${\G}$-action $L''\simeq L_{\alpha}''$. 
Therefore 
\begin{eqnarray*}
& & \partial^{L'}_{L''} ({\Res}^{L}_{L'} (\omega)) \\ 
& = & 
\sum_{i>j} (-1)^{i+j} \sum_{\alpha} [{\G}_{L_{\alpha}''}:{\G}_{L_{\alpha}''/L'}] \cdot 
\{ \nu_{L'}(f_i) \nu_{L_{\alpha}''}(f_{j}|_{L'}) - \nu_{L'}(f_j) \nu_{L_{\alpha}''}(f_{i}|_{L'}) \} 
\cdot \omega_{i,j}. 
\end{eqnarray*}

On the other hand, 
by substituting $f_i\wedge f_j$ in $\omega$, 
we find that 
\begin{equation*}
\partial^{L'}_{L''} ({\Res}^{L}_{L'} (f_i \wedge f_j))   =  
\sum_{\alpha} [{\G}_{L_{\alpha}''}:{\G}_{L_{\alpha}''/L'}] \cdot 
\{ \nu_{L'}(f_i) \nu_{L_{\alpha}''}(f_{j}|_{L'}) - \nu_{L'}(f_j) \nu_{L_{\alpha}''}(f_{i}|_{L'}) \}. 
\end{equation*}
It follows that 
\begin{equation*}
\partial^{L'}_{L''} ({\Res}^{L}_{L'} (\omega)) = 
\sum_{i>j} (-1)^{i+j} \partial^{L'}_{L''} ({\Res}^{L}_{L'} (f_i \wedge f_j)) \cdot \omega_{i,j}. 
\end{equation*}

Finally, taking the sum over $L'\in L^{(1)}$ with $Z_{L''}\subset Z_{L'} \subset \tilde{X}_{L}$, 
we obtain 
\begin{eqnarray*}
\sum_{L'} \partial^{L'}_{L''} ({\Res}^{L}_{L'} (\omega))  
& = & \sum_{L'} \sum_{i>j} (-1)^{i+j} \partial^{L'}_{L''} ({\Res}^{L}_{L'} (f_i \wedge f_j)) \cdot \omega_{i,j} \\ 
& = & \sum_{i>j} (-1)^{i+j} \left( \sum_{L'}  \partial^{L'}_{L''} ({\Res}^{L}_{L'} (f_i \wedge f_j)) \right) \cdot \omega_{i,j} \\ 
& = & 0,  
\end{eqnarray*}
where the last equality follows from Step 1. 
\end{proof}

The next step is to allow ${\G}$ to be general 
while keeping $L$ to be $L_0$. 

\begin{step}
Proposition \ref{prop: complex} holds in the case 
$L=L_0$, ${\Ohat}(L)<{\G}<{\rm O}^{+}(L)$. 
\end{step}

\begin{proof}
We consider the diagram 
\begin{equation*}
\xymatrix{
\wedge^{q}M(L) \ar[r]^-{{\Res}} \ar[d] & 
\bigoplus_{L'} \wedge^{q-1}M(L')^{\tilde{G}_{L'}} \ar[d] \ar[r]^-{\partial} &  
\bigoplus_{\alpha} \wedge^{q-2}M(L_{\alpha}'')^{\tilde{G}_{L_{\alpha}''}} \ar[d] \\ 
\wedge^{q}M(L)^{G_{L}} \ar[r]^-{{\Res}} & 
\bigoplus_{L'} \wedge^{q-1}M(L')^{G_{L'}} \ar[r]^-{\partial} &  
\wedge^{q-2}M(L'')^{G_{L''}}
}
\end{equation*}
where the upper and lower sequences are 
the Borcherds-Gersten complexes of ${\Ohat}(L)$ and ${\G}$ respectively, 
$L_{\alpha}''$ ranges over ${\Otilde}(L)$-equivalence classes of sublattices of $L$ which are ${\G}$-equivalent to $L''$, 
and the vertical maps are transfer maps. 
By Proposition \ref{prop: residue and transfer}, the left square is commutative. 
Similarly, the right square also commutes. 

By Step 2, we know that the composition in the upper sequence is zero. 
Since the transfer map in the left is surjective, 
this implies that the composition in the lower sequence is also zero. 
\end{proof}

Before going to the general case, 
it will be useful to recall or prepare some notation. 
For two primitive sublattices $L_1\subset L_2$ of $L_0$, 
we denote by $\mathcal{C}(L_1, L_2; {\G})$ 
the set of ${\G}_{L_2}$-equivalence classes of sublattices $L_{1}' $ of $L_2$ 
which are ${\G}$-equivalent to $L_1$. 
We write ${\G}_{L_1/L_2}$ for the image of 
the natural map ${\rm Stab}_{{\G}_{L_2}}(L_1)\to {\rm O}^{+}(L_1)$, 
and write $G_{L_1/L_2} = {\G}_{L_1/L_2}/{\Ohat}(L_1)$. 
We denote by $Z_{L_{1}/L_{2}}\subset X_{L_2}$ the special cycle for $L_1 \subset L_2$ in $X_{L_2}$ (not in $X$). 
We also write 
$X_{L_{1}/L_{2}} = X_{{\G}_{L_1/L_2}}$. 
Then $X_{L_{1}/L_{2}}$ is the normalization of $Z_{L_{1}/L_{2}}$. 

\begin{step}
Proposition \ref{prop: complex} holds in the full generality. 
\end{step}

\begin{proof}
For each $L'$ with $L''\subset L' \subset L$, we write 
\begin{equation*}
\mathcal{C}(L', L; {\G}) = \{ L_{\alpha}' \}_{\alpha}, \quad G_{\alpha}=G_{L_{\alpha}'/L}, 
\end{equation*}
\begin{equation*}
\mathcal{C}(L'', L'; {\G}) = \{ L_{\beta}'' \}_{\beta}, \quad G_{\beta}=G_{L_{\beta}''/L'}.  
\end{equation*}
For each $\alpha$, we choose an isometry $L_{\alpha}'\simeq L'$ given by ${\G}$-action. 
If we view ${\G}_{L_{\alpha}'/L}<{\rm O}^{+}(L_{\alpha}')$ as a subgroup of ${\rm O}^{+}(L')$ 
via $L_{\alpha}'\simeq L'$, we have  
\begin{equation}\label{eqn: inclusion groups I}
{\Ohat}(L') < {\G}_{L_{\alpha}'/L} < {\G}_{L'} < {\rm O}^{+}(L'). 
\end{equation}
We denote by $L_{\alpha}''\subset L_{\alpha}'$ the image of $L''\subset L'$ by $L'\simeq L_{\alpha}'$. 
Then we write   
\begin{equation*}
\mathcal{C}(L_{\alpha}'', L_{\alpha}'; {\G}_{L}) = \{ L_{\alpha, \gamma}'' \}_{\gamma}, \quad 
G_{\alpha, \gamma}={\G}_{L_{\alpha,\gamma}''/L_{\alpha}'/L}/{\Ohat}(L_{\alpha,\gamma}''), 
\end{equation*}
where ${\G}_{L_{\alpha,\gamma}''/L_{\alpha}'/L}$ is the image of 
${\rm Stab}_{{\G}_{L_{\alpha}'/L}}(L_{\alpha,\gamma}'') \to {\rm O}^{+}(L_{\alpha,\gamma}'')$.  
By sending $L_{\alpha,\gamma}''\subset L_{\alpha}'$ back by $L_{\alpha}'\simeq L'$ 
and remembering its ${\G}_{L'}$-equivalence class, 
we have an assignment $\gamma\mapsto \beta$ with 
$L_{\alpha,\gamma}'' \stackrel{{\G}_{L'}}{\sim} L_{\beta}''$. 
Under the identification $L_{\alpha,\gamma}'' \simeq L_{\beta}'' \simeq L''$, 
we have 
\begin{equation}\label{eqn: inclusion groups II}
{\Ohat}(L'') < {\G}_{L_{\alpha,\gamma}''/L_{\alpha}'/L} < {\G}_{L_{\beta}''/L'} < {\G}_{L''} < {\rm O}^{+}(L''). 
\end{equation}

With these notations, 
the map $\partial^{L}_{L'} \circ \partial^{L'}_{L''}$ can be factorized as follows: 
\begin{equation}\label{eqn: CD step 4}
\xymatrix{
\wedge^{q}M(L)^{G_L} \ar[r]^-{{\Res}} \ar[rd]_{\partial^{L}_{L'}} & 
\bigoplus_{\alpha}\wedge^{q-1}M(L_{\alpha}')^{G_{\alpha}} \ar[d] \ar[r]^-{{\Res}} & 
\bigoplus_{\alpha, \gamma}\wedge^{q-2}M(L_{\alpha, \gamma}'')^{G_{\alpha,\gamma}} \ar[d] \\ 
 &  \wedge^{q-1}M(L')^{G_{L'}} \ar[r]^-{{\Res}} \ar[rd]_{\partial^{L'}_{L''}} & 
\bigoplus_{\beta} \wedge^{q-2}M(L_{\beta}'')^{G_{\beta}} \ar[d] \\ 
 & & \wedge^{q-2}M(L'')^{G_{L''}}
}
\end{equation}
Here the two triangles are the defining factorization of each boundary maps, 
the middle vertical map is the transfer map for \eqref{eqn: inclusion groups I}, 
and the right vertical maps are the transfer maps for \eqref{eqn: inclusion groups II}. 
Commutativity of the square follows from Proposition \ref{prop: residue and transfer}. 

On the other hand, we consider 
\begin{equation*}
\mathcal{C}(L'', L; {\G}) = \{ L_{\delta}'' \}_{\delta}, \quad G_{\delta}=G_{L_{\delta}''/L}. 
\end{equation*}
By forgetting the intermediate lattice $L_{\alpha}'$ in 
$L_{\alpha, \gamma}''\subset L_{\alpha}' \subset L$, 
we have an assignment $(\alpha, \gamma)\mapsto \delta$ with 
$L_{\alpha, \gamma}'' \stackrel{{\G}_L}{\sim} L_{\delta}''$. 
Under the identification 
$L_{\alpha, \gamma}'' \simeq L_{\delta}'' \simeq L''$, 
we have 
\begin{equation}\label{eqn: inclusion groups III}
{\Ohat}(L'') < {\G}_{L_{\alpha, \gamma}''/L_{\alpha}'/L} < {\G}_{L_{\delta}''/L} < {\G}_{L''} < {\rm O}^{+}(L''). 
\end{equation}
By comparing \eqref{eqn: inclusion groups II} and \eqref{eqn: inclusion groups III}, 
we see that the composition of the two right vertical maps in \eqref{eqn: CD step 4} factorizes as 
\begin{equation*}
\xymatrix{
\bigoplus_{\alpha, \gamma}\wedge^{q-2}M(L_{\alpha, \gamma}'')^{G_{\alpha,\gamma}}  \ar[r] \ar[d] & 
\bigoplus_{\delta}\wedge^{q-2}M(L_{\delta}'')^{G_{\delta}} \ar[d] \\ 
\bigoplus_{\beta}\wedge^{q-2}M(L_{\beta}'')^{G_{\beta}} \ar[r]  & \wedge^{q-2}M(L'')^{G_{L''}}
}
\end{equation*}
For each $\alpha$, the composition 
\begin{equation*}
\wedge^{q-1}M(L_{\alpha}')^{G_{\alpha}} \stackrel{{\Res}}{\longrightarrow} 
\bigoplus_{\gamma} \wedge^{q-2}M(L_{\alpha, \gamma}'')^{G_{\alpha, \gamma}} \longrightarrow 
\bigoplus_{\delta} \wedge^{q-2}M(L_{\delta}'')^{G_{\delta}} 
\end{equation*}
is the direct sum over $\delta$ of the boundary map from 
$Z_{L_{\alpha}'/L}$ to $Z_{L_{\delta}''/L}$ 
in the Borcherds-Gersten complex for $X_{L}$. 
This shows that $\partial_{L''}^{L'} \circ \partial_{L'}^{L}$ factors through 
\begin{equation*}
\wedge^{q}M(L)^{G_{L}} \stackrel{{\Res}}{\longrightarrow} 
\bigoplus_{\alpha}\wedge^{q-1}M(L_{\alpha}')^{G_{\alpha}} \longrightarrow 
\bigoplus_{\delta} \wedge^{q-2}M(L_{\delta}'')^{G_{\delta}}, 
\end{equation*}
where the maps are the boundary maps in the Borcherds-Gersten complex for $X_L$. 

Finally, taking the sum over $L'\in L_0^{(r+1)}$ with 
$Z_{L''}\subset Z_{L'} \subset Z_{L}$, 
we see that 
$\partial \circ \partial$ factors through 
\begin{equation}\label{eqn: step 4}
\wedge^{q}M(L)^{G_{L}} \stackrel{{\Res}}{\longrightarrow} 
\bigoplus_{L'}\wedge^{q-1}M(L')^{G_{L'/L}} \longrightarrow 
\bigoplus_{\delta} \wedge^{q-2}M(L_{\delta}'')^{G_{\delta}}, 
\end{equation}
where now $L'$ ranges over ${\G}_{L}$-equivalence classes of corank $1$ sublattices of $L$ 
with $Z_{L_{\delta}''/L} \subset Z_{L'/L}$ for some $\delta$. 
This is the sum over $\delta$ 
of the Borcherds-Gersten complex for $X_{L}$ 
from $L$ to $L_{\delta}''\subset L$.   
By Step 3, we find that the composition is zero. 
\end{proof}

The proof of Proposition \ref{prop: complex} is now completed. 
\end{proof}

The geometry behind the argument in Step 4 is the commutative diagram 
\begin{equation*}
\xymatrix{
\bigsqcup_{\alpha,\gamma}X_{\alpha,\gamma} \ar[r]^-{\pi} \ar[ddd] & 
\bigsqcup_{\alpha,\gamma}Z_{\alpha,\gamma} \ar[d] \ar[r] \ar@/_24pt/[ddd] & 
\bigsqcup_{\alpha}X_{\alpha} \ar[d]_{\pi} \ar@/^24pt/[ddd] &  \\ 
 & \sum_{\delta}Z_{\delta} \ar@{^{(}->}[r] \ar[d] & \sum_{\alpha}Z_{\alpha} \ar@{^{(}->}[r] \ar[d] & X_{L}  \ar[d]^{\pi} \\ 
 & Z_{L''} \ar@{^{(}->}[r] & Z_{L'} \ar@{^{(}->}[r] & Z_{L} \\ 
\bigsqcup_{\beta}X_{\beta} \ar[r]^{\pi} & \sum_{\beta}Z_{\beta} \ar@{^{(}->}[r] \ar[u] & X_{L'} \ar[u]^{\pi} &   
}
\end{equation*}
Here $Z_{\alpha}=Z_{L_{\alpha}'/L}$, 
$X_{\alpha}=X_{L_{\alpha}'/L}$, 
$Z_{\alpha,\gamma}\subset X_{\alpha}$ is the special cycle for $L_{\alpha,\gamma}'' \subset L_{\alpha}'$, 
and so on. 
Various $\pi \colon X_{?}\to Z_{?}$ are normalization maps. 
The essence of the argument in Step 4 is just 
passage from the line 
$Z_{L''}\hookrightarrow Z_{L'}\hookrightarrow Z_{L}$ 
to the upper line 
$\sum_{\delta}Z_{\delta} \hookrightarrow \sum_{\alpha}Z_{\alpha} \hookrightarrow X_{L}$. 
Thus, simply speaking,  
Step 4 is the process of taking the normalization $X_L\to Z_L$. 
Similarly, Step 3 is the process of taking the finite cover $\tilde{X}_{L}\to X_{L}$.

\section{Regulators and regularized theta lifts}\label{sec: regulator}

A standard approach in studying higher cycles is to calculate their regulators. 
In this section we gather some remarks on the regulators of special higher cycles, especially of $(p, 1)$-type. 
This is a long supplement to \S \ref{ssec: special higher cycle}. 
What we do is no more than recalling general facts, 
but nevertheless it could be useful to make this explicit. 

\subsection{Regulators}\label{ssec: regulator}

Let $X$ be a smooth projective variety of dimension $n$, 
and $A={\Q}$ or ${\R}$. 
According to Beilinson \cite{Be} and Bloch \cite{Bl}, we have a cycle map 
\begin{equation}\label{eqn: regulator}
{\CH}^{p}(X, m)_{{\Q}} \to H^{2p-m}_{\mathcal{D}}(X, A(p)) 
\end{equation}
to the Deligne cohomology, usually called the \textit{regulator map}. 
Some explicit formulae for this map have been known. 
We refer to \cite{KLM} and the references therein for comprehensive accounts. 
Here, for simplicity, we consider only the case $m=1$ and $A={\R}$. 
The real Deligne cohomology is described as 
\begin{equation*}
H^{2p-1}_{\mathcal{D}}(X, {\R}(p)) \simeq 
\frac{H^{2p-2}(X, {\C})}{F^{p}+H^{2p-2}(X, {\R}(p))} \simeq 
H^{p-1,p-1}(X, {\R}(p-1)). 
\end{equation*}
This can be regarded as a dual of $H^{n-p+1,n-p+1}(X, {\R})$. 
If $\sum_{i}(Z_i, \psi_i)$ is a $(p, 1)$-cycle, 
its real regulator is the functional 
\begin{equation}\label{eqn: regulator (p,1)}
\omega \mapsto (2\pi \sqrt{-1})^{p-n-1} \sum_{i} \int_{Z_i}\log |\psi_{i}| \cdot \omega 
\end{equation}
on $H^{n-p+1,n-p+1}(X, {\R})$. 


\subsection{Shimura surfaces}

The modular variety $X={\D}_{L}/{\G}$ is compact only when $L$ is anisotropic,  
and this happens only when $n\leq 2$. 
So let $L$ be anisotropic of signature $(2, 2)$ and suppose that ${\G}$ is torsion-free. 
Then $X={\D}_{L}/{\G}$ is a smooth projective surface, 
known as a \textit{quaternionic Shimura surface}. 
Let $\sum_{i}(Z_{L_i}, \psi(f_i))$ be a special $(2, 1)$-cycle on $X$.  
By \eqref{eqn: regulator (p,1)} and \eqref{eqn: regularized theta lift}, 
the real regulator of this cycle is given, up to constant, by the functional 
\begin{equation}\label{eqn: regulator regularized theta}
\omega \: \mapsto \: \sum_{i} \int_{Z_{L_i}} \Phi(f_i) \cdot \omega 
\end{equation}
on $H^{1,1}(X, {\R})$, 
where $\Phi(f)$ is the regularized theta lift of $f$. 
When $\omega$ is the class of a special curve $Y\subset X$ 
which intersects properly with $\sum_i (Z_{L_i}, \psi(f_i))$, 
this pairing value can be written as 
\begin{equation*}
\sum_{i} \sum_{p\in Y\cap Z_{L_i}} \Phi(f_i)(p). 
\end{equation*}
In this way, CM values of regularized theta lifts naturally show up in the regulator formula.

\subsection{On the non-compact case}

When $L$ is general, $X={\D}_{L}/{\G}$ is no longer compact. 
There are at least two approaches in this situation: 
\begin{enumerate}
\item To take a (toroidal) compactification $X\hookrightarrow \bar{X}$ 
and study extension of the cycles. 
\item To study the regulator map \eqref{eqn: regulator} in the non-compact case. 
\end{enumerate}
In the rest of this section, we extend our remark for each approach.

\subsubsection{Extension}

We take a toroidal compactification $X\hookrightarrow \bar{X}$ 
and let $D=\bar{X}-X$ be the boundary divisor.  
Let $\xi=\sum_{i}(Z_{L_i}, \psi(f_i))$ be a special $(p, 1)$-cycle on $X$. 
By the localization exact sequence 
\begin{equation*}
\cdots \to {\CH}^{p}(\bar{X}, 1)_{{\Q}} \to {\CH}^{p}(X, 1)_{{\Q}} \stackrel{\delta}{\to} 
{\CH}^{p-1}(D)_{{\Q}} \to {\CH}^{p}(\bar{X})_{{\Q}} \to \cdots 
\end{equation*}
we see that $\xi$ is extendable over $\bar{X}$ if and only if $\delta(\xi)=0$. 
The connecting map $\delta$ is described as follows. 
Let $\bar{Z}_i$ be the closure of $Z_{L_i}$ in $\bar{X}$ 
and $\bar{\psi}_i$ be $\psi(f_i)$ viewed as a rational function on $\bar{Z}_{i}$. 
Suppose that $\bar{\xi}=\sum_{i}(\bar{Z}_i, \bar{\psi}_i)$ intersects properly with $D$. 
The cycle $\sum_{i}{\rm div}(\bar{\psi}_{i})$ is supported on $D$ 
by the cocycle condition for $\xi$. 
Then $\delta(\xi)$ is the class of $\sum_{i}{\rm div}(\bar{\psi}_{i})$ in ${\CH}^{p-1}(D)_{{\Q}}$. 

If $\sum_{i}{\rm div}(\bar{\psi}_{i})=0$ as a cycle 
(not just rationally equivalent to $0$), 
then $\bar{\xi}$ already defines a $(p, 1)$-cycle on $\bar{X}$, 
and this gives an apparent and tractable extension of $\xi$. 
A typical situation where this holds is the case 
each ${\rm div}(\bar{\psi}_{i})$ has no extra component contained in $D$. 

\subsubsection{Non-compact regulator}

The regulator map \eqref{eqn: regulator} is defined even when $X$ is non-compact, 
with values in the Deligne-Beilinson cohomology. 
We take $A={\Q}$. 
The composition of \eqref{eqn: regulator} (with $m=1$) 
and the natural map to the singular cohomology takes the form  
\begin{equation}\label{eqn: non-cpt regulator}
{\CH}^{p}(X, 1)_{{\Q}} \longrightarrow (2\pi \sqrt{-1})^{p} W_{2p}H^{2p-1}(X, {\Q})\cap F^p, 
\end{equation}
where $(W_{\bullet}, F^{\bullet})$ is the weight and Hodge filtrations 
in the mixed Hodge structure on $H^{\ast}(X)$ (see \cite{Be2}). 

In our modular situation, we have the following vanishing: 

\begin{lemma}\label{lem: vanishing discrete part}
Let $X={\D}_{L}/{\G}$. 
The target of \eqref{eqn: non-cpt regulator} vanishes when $p<n/2$. 
\end{lemma}

\begin{proof}
The mixed Hodge structure on $H^{k}(X, {\Q})$ is pure when $k<n-1$ (see, e.g., \cite{BLMM}). 
Therefore, when $p<n/2$, 
the target of \eqref{eqn: non-cpt regulator} is 
the intersection of the rational part and the $F^{p}$-part of a pure ${\Q}$-Hodge structure of weight $2p-1$, 
which clearly vanishes. 
\end{proof}

Thus the map \eqref{eqn: non-cpt regulator} can be nonzero only when $p\geq n/2$. 
In this case, by the general formula, 
the cycle class of a special $(p, 1)$-cycle 
$\sum_{i}(Z_{L_i}, \psi(f_i))$ is given up to constant by 
\begin{equation}\label{eqn: non-cpt cycle class}
\sum_{i} \int_{Z_{L_i}} d \log \psi(f_i) \wedge \bullet \; \; 
\in W_{2p}H^{2p-1}(X, {\Q})\cap F^p. 
\end{equation}
Since $W_{2p-1}H^{2p-1}(X, {\Q})\cap F^p = 0$, 
the projection 
\begin{equation*}\label{eqn: MHS Gr}
W_{2p}H^{2p-1}(X, {\Q})\cap F^p \to {\rm Gr}^{W}_{2p}H^{2p-1}(X, {\Q})\cap F^p
\end{equation*}
is injective. 
So the class \eqref{eqn: non-cpt cycle class} is determined by its image 
in ${\rm Gr}^{W}_{2p}H^{2p-1}(X, {\Q})$, 
which is a Hodge class. 

On the other hand, when $p<n/2$, 
Lemma \ref{lem: vanishing discrete part} shows that 
the map \eqref{eqn: regulator} takes the form 
\begin{equation}\label{eqn: non-cpt regulator II}
{\CH}^{p}(X, 1)_{{\Q}} \to \frac{H^{2p-2}(X, {\C})}{H^{2p-2}(X, {\Q}(p))+F^{p}} 
\end{equation}
It seems that little is known for such a non-compact Abel-Jacobi map. 
In the case $p-1<n/4$, 
$H^{2p-2}(X)$ is generated by special cycles (\cite{BLMM}). 
Then \eqref{eqn: non-cpt regulator II} can be written as 
\begin{equation*}
{\CH}^{p}(X, 1)_{{\Q}} \to H^{2p-2}(X, {\Q}(p-1))\otimes_{{\Q}} {\C}/{\Q}(1). 
\end{equation*}
It would be interesting to express this non-compact Abel-Jacobi invariant of special $(p, 1)$-cycles  
in terms of special cycles and Borcherds products.

\begin{remark}
Even when $X$ is non-compact, \eqref{eqn: regulator (p,1)} still makes sense as a current, 
though its relation to the Beilinson regulator is not clear to me. 
In the case $p=2$, 
currents of the form \eqref{eqn: regulator regularized theta} for general $L$ 
have relevance to the theta lifting (\cite{Ga}). 
\end{remark}



\begin{thebibliography}{99}

\bibitem{Be}Beilinson, A.~A. 
\textit{Higher regulators and values of L-functions.}  
J. Soviet Math. \textbf{30} (1985) 2036--2070. 

\bibitem{Be2}Beilinson, A.~A. 
\textit{Notes on absolute Hodge cohomology.}  
in ``Applications of algebraic $K$-theory to algebraic geometry and number theory (Boulder, 1983)'', 
Contemp. Math., \textbf{55} (1986), 35--68. 


\bibitem{BLMM}Bergeron, N.; Li, Z.; Millson, J.; Moeglin, C. 
\textit{The Noether-Lefschetz conjecture and generalizations.} 
Invent. Math. \textbf{208} (2017), no. 2, 501--552. 


\bibitem{Bl}Bloch, S. 
\textit{Algebraic cycles and higher $K$-theory.}
Adv. in Math. \textbf{61} (1986), no.3, 267--304.


\bibitem{Bo95}Borcherds, R. 
\textit{Automorphic forms on $O_{s+2,2}({\R})$ and infinite products.} 
Invent. Math. \textbf{120} (1995), no. 1, 161--213. 

\bibitem{Bo98}Borcherds, R. 
\textit{Automorphic forms with singularities on Grassmannians.} 
Invent. Math. \textbf{132} (1998), no. 3, 491--562. 

\bibitem{BKPSB}
Borcherds, R.; Katzarkov, L.; Pantev, T.; Shepherd-Barron, N. I. 
\textit{Families of K3 surfaces.} 
J. Algebraic Geom. \textbf{7} (1998), no. 1, 183--193. 

\bibitem{Br}Bruinier, J.~H.
\textit{Borcherds products on $O(2, l)$ and Chern classes of Heegner divisors.} 
Lecture Notes in Math. \textbf{1780}, Springer-Verlag, 2002. 

\bibitem{BEF}Bruinier, J.~H.; Ehlen, S.; Freitag, E. 
\textit{Lattices with many Borcherds products.} 
Math. Comp. \textbf{85} (2016), 1953--1981. 

\bibitem{BR}Bruinier, J.~H.; Westerholt-Raum, M. 
\textit{Kudla's modularity conjecture and formal Fourier-Jacobi series.} 
Forum Math. Pi \textbf{3} (2015), e7, 30 pp.

\bibitem{Ga}Garcia, L.~E. 
\textit{Regularized theta lifts and $(1,1)$-currents on GSpin Shimura varieties.} 
Algebra Number Theory \textbf{10} (2016), no.3, 597--644.

\bibitem{GS}Gille, P.; Szamuely, T.  
\textit{Central simple algebras and Galois cohomology.} 
2nd edition. Cambridge Univ.~Press, 2017. 


\bibitem{Go}Goncharov, A. 
\textit{Regulators.} 
in ``Handbook of $K$-theory, I", 295--349, Springer, 2005. 

\bibitem{GN}Gritsenko, V.~A.; Nikulin, V.~V. 
\textit{Automorphic forms and Lorentzian Kac-Moody algebras. II.}  
Internat. J. Math. \textbf{9} (1998), no. 2, 201--275. 

\bibitem{HMP}Howard, B.; Madapusi Pera, K. 
\textit{Arithmetic of Borcherds products.} 
in ``Arithmetic divisors on orthogonal and unitary Shimura varieties.'' 
Ast\'erisque \textbf{421} (2020), 187--297.

\bibitem{Ka}Kato, K. 
\textit{Milnor $K$-theory and the Chow group of zero cycles.} 
in ``Applications of algebraic $K$-theory to algebraic geometry and number theory (Boulder, 1983)'', 
Contemp. Math., \textbf{55} (1986), 241--253. 

\bibitem{KLM}Kerr, M.; Lewis, J.; M\"uller-Stach, S. 
\textit{The Abel-Jacobi map for higher Chow groups.}
Compos. Math. \textbf{142} (2006), no.2, 374--396.


\bibitem{Ma}Ma, S. 
\textit{Quasi-pullback of Borcherds products.}  
Bull. Lond. Math. Soc. \textbf{51} (2019), no. 6, 1061--1078. 



\bibitem{Ni}Nikulin, V.V. 
\textit{Integral symmetric bilinear forms and some of their applications.}
Math. USSR Izv. \textbf{14} (1980), 103--167.  

\bibitem{Ra1}Ramakrishnan, D. 
\textit{Valeurs de fonctions L des surfaces d'Hilbert-Blumental en $s=1$.} 
C. R. Acad. Sci. Paris Math. \textbf{301} (1985), no.18, 809--812. 

\bibitem{Ra2}Ramakrishnan, D. 
\textit{Arithmetic of Hilbert-Blumenthal surfaces.} 
in ``Number theory (Montreal, 1985)'', 
CMS Conf. Proc. \textbf{7} (1987), 285--370. 


\bibitem{Ra3}Ramakrishnan, D. 
\textit{Modular curves, modular surfaces, and modular fourfolds.} 
in ``Algebraic cycles and motives. Vol. 1'', 
LMS Lecture Note Ser., \textbf{343} (2007), 278--292. 


\bibitem{SS}Saito, S.; Sato, K. 
\textit{Algebraic cycles and \'etale cohomology.} 2nd edition. (Japanese) 
Maruzen, 2025. 



\bibitem{Sr2}Sreekantan, R. 
\textit{Algebraic cycles and values of Green's functions.} 
arXiv:2022.08325

\bibitem{Su}Suslin, A.~A. 
\textit{Reciprocity laws and the stable rank of polynomial rings.} 
Math. USSR. Izv. \textbf{15} (1980), 589--623. 


\bibitem{Ze}Zemel, S. 
\textit{Seesaw identities and theta contractions with generalized theta functions, and restrictions of theta lifts.} 
Ramanujan J. \textbf{63} (2024), no.3, 749--771. 

\bibitem{Zhang}Zhang, W. 
\textit{Modularity of generating functions of special cycles on Shimura varieties.} 
Thesis, Columbia University, 2009. 
 

\end{thebibliography}
\end{document}